\documentclass[12pt]{amsart}

\usepackage{subfig,graphicx,latexsym,amssymb}
\usepackage[colorlinks=true]{hyperref}
\usepackage[width=15cm, height=19cm]{geometry}
\usepackage[toc]{appendix}
\usepackage{ulem}
\usepackage{xcolor}
\graphicspath{ {Figures/} }

\newtheorem{theorem}{Theorem}[section]

\newtheorem{lemma}[theorem]{Lemma}

\newtheorem{remark}[theorem]{Remark}

\allowdisplaybreaks

\newcommand{\bR}{{\bf R}}
\newcommand{\R}{\mathbb{R}}

\newcommand{\bQ}{{\bf Q}}

\newcommand{\cF}{\mathcal{F}}

\newcommand{\cB}{\mathcal{B}}

\newcommand{\cL}{\mathcal{L}}

\numberwithin{equation}{section}

\newcommand{\N}{{\mathbb{N}}}
\newcommand{\E}{{\mathbb{E}}}

\newcommand{\bW}{\mathbf{W}}
\newcommand{\bwW}{\widetilde{\mathbf{W}}}
\newcommand{\boW}{\overline{\mathbf{W}}}
\newcommand{\wW}{\widetilde{W}}
\newcommand{\oW}{\overline{W}}
\newcommand{\oGamma}{\overline{\Gamma}}
\newcommand{\wN}{\widetilde{N}}
\newcommand{\wV}{\widetilde{V}}
\newcommand{\wGamma}{\widetilde{\Gamma}}

\begin{document}

\title{Foreign exchange options on Heston-CIR model under L\'{e}vy process framework}

\author{Giacomo Ascione $^{1}$}
\author{Farshid Mehrdoust $^{2}$}
\author{Giuseppe Orlando  $^{*,3}$}
\author{Oldouz Samimi $^{2}$}

\thanks{%
\noindent
    $^{1}$
    Scuola Superiore Meridionale, Largo S. Marcellino, 10, 80138 Napoli, Italy\\
	$^{2}$
	Department of Applied Mathematics, Faculty of Mathematical Sciences,
	University of Guilan, P. O. Box:~41938-1914, Rasht, Iran,\\
	$^{3}$
	Department of Economics and Finance,
	University of Bari ``Aldo Moro'', Largo Abbazia S.~Scolastica, I-70124 Bari, Italy}
\thanks{$^{*}$ Corresponding author}

\begin{abstract}
In this paper, we consider the Heston-CIR model with L\'{e}vy process for pricing in the foreign exchange (FX) market by providing a new formula that better fits the distribution of prices. To do that, first, we study the existence and uniqueness of the solution to this model. Second, we examine the strong convergence of the L\'{e}vy process with stochastic domestic short interest rates, foreign short interest rates and stochastic volatility. Then, we apply Least Squares Monte Carlo (LSM) method for pricing American options under our model with stochastic volatility and stochastic interest rate. Finally, by considering real-world market data, we illustrate numerical results for the four-factor Heston-CIR L\'{e}vy model. 
\end{abstract}

\maketitle

\markboth{On Heston-CIR diffusion models under L\'{e}vy process}{G.~Ascione, F.~Mehrdoust, G.~Orlando, O.~Samimi}
\sloppy
\begin{small}



\end{small}

\section{Introduction} 

In this paper, we study the problem of pricing American options in foreign exchange (FX) markets.  Our purpose is to consider a Heston hybrid model, that is the Heston-CIR model. The said model correlates stochastic volatility and stochastic interest rate, under the assumption that the underlying stock returns follow a L\' evy process. In FX markets, option pricing under stochastic interest rate and stochastic volatility was first introduced by Grzelak et al. \cite{Grzelak} and Van Haastrecht et al. \cite{vanHaastrecht2011}.

Let $(\Omega, \cF, \{\cF\}_{t \geq 0}, \bQ)$ be a filtered probability space with the risk-neutral probability $\bQ$ on which a Brownian motion is defined $W= \{W_t\}_{t \geq 0}$. Consider an asset price process  $\{S_t\}_{t \geq 0}$ following a Geometric Brownian Motion (GBM) process and satisfying the following stochastic differential equation (SDE)
\begin{equation} \label{Eq:GBM}
dS_t = r S_t \ dt + \sigma S_t \ dW_t, \ t \geq 0,
\end{equation}
where $S_t$ is the asset price at time $t$ and the constant parameters $r>0$ and $\sigma>0$ are, respectively, the (domestic) risk-neutral interest rate and the volatility. This model was presented by Black-Scholes and Merton \cite{Merton} and was a major step in arbitrage-free option pricing because it evaluates options at the risk-neutral rate regardless of the risk and return of the underlying. Meanwhile the model \eqref{Eq:GBM} became very popular in finance among practitioners because it takes positive values and it has simple calculations (for example, see \cite{Merton}, \cite{Higham}, \cite{Higham3}, \cite{Steven},  \cite{Orlando}, \cite{Orlando2}, \cite{Orlando3}, \cite{Orlando4}). But along with all these advantages, this model has also some noteworthy shortcomings such as constant interest rate, constant volatility and absence of jump term.

The need for more sophisticated frameworks, to cope with the above-mentioned shortfalls, has led the ensuing literature to the development of a number of papers for pricing  derivatives on risky assets based on stochastic asset price models generalizing the classical GBM paradigm (\cite{Hull}, \cite{Stein}, \cite{Heston}, \cite{Benhamou}, \cite{Mikhailov}, \cite{Fallah}, \cite{Teng})

For example, with regard to interest rates, the most widely used model in finance is the Cox-Ingersoll-Ross (CIR) model which assumes the risk-neutral dynamics of the instantaneous interest rate to be described by a stochastic process $\{r_t\}_{t\ge 0}$  (also known as {\it square-root process} ) driven by the SDE
\[
dr_t  = a (b - r_t) \ dt + \sigma \sqrt{r_t} \ dW_t, 
\]
where $a$ is the speed of adjustment of interest rates to the long-term mean $b$ and $\sigma$ is the volatility. 

Further, concerning volatility, one of the pioneering papers is that of Steven L. Heston \cite{Heston}  who derived the pricing formula of a stock European option when the dynamics of the underlying stock price are described by a model with non-constant volatility supposed to be stochastic.

Given the popularity and advantages of both the Heston and the CIR models, the hybrid version of them is the so-called Heston-CIR model \cite{Ahlip, Grzelak, Grzelak1, Grzelak2, Van} (see also \cite{Fallah} for application to the American option pricing), where the dynamics of the underlying asset price $S_t$ is given by
\begin{equation}\label{HCmodel}
\begin{split}
& dS_t  = S_t (r_t - \delta) \ dt + S_t \sqrt{V_t}\, dW^1_t  , \quad S_0 > 0    \\  
& dV_t  = \kappa_v (\theta_v- V_t) \ dt + \sigma_v \sqrt{V_t}\, dW^2_t, \quad V_0 >0 \\
& dr_t  = \kappa_d (\theta_r - r_t) \ dt + \sigma_d \sqrt{r_t} \, dW^3_t, \quad r_0 >0,
\end{split}
\end{equation}
where  $V_t$ and $r_t$ are, respectively, the stochastic variance and the stochastic domestic interest rate of the stock return.  All processes are defined under the domestic risk-neutral measure, $\bQ$. $\sigma_v$ is the second-order volatility, i.e.  the volatility of variance (often called the {\it volatility of volatility} or, shorter, {\it vol of vol}),  and $\sigma_d$ denotes the volatility of the short rate $r_t$. $\kappa_v$ and $\kappa_d$ are the mean reverting rates of the variance and short rate processes, respectively. We assume the long-run mean of the asset price $\delta$ is constant. The parameters $\theta_v$ and $\theta_d$ are respectively, the long-run mean of the variance and interest rate. $ S_0, V_0, r_0$ denote the initial asset price,  variance and interest rate. The standard Brownian motions $\{W_t^i\}_{t\ge 0}$, for $ i=1, 2,3$, are supposed to be correlated. The correlation matrix  is given by
\begin{align}
\begin{pmatrix} 1 & \rho_{sv} & \rho_{sd} \\  \rho_{sv} & 1  & \rho_{vd} \\ \rho_{sd} & \rho_{vd} & 1      \end{pmatrix}. \nonumber
\end{align}
\subsection{Jumps in financial models}
With regard to jumps, there are lots of reasons to utilize them in financial models. In the real world, asset price empirical distributions present fat tails and high peaks and are asymmetric, behaviour that deviates from normality. From a risk management perspective, jumps allow quantifying and taking into account the risk of strong stock price movements over short time intervals, which appears non-existent in models with continuous paths.  Anyway, the strongest argument for using discontinuous models is simply the presence of jumps in observed prices. Thus, we want to model these phenomena with jump diffusion or L\'evy processes. A jump-diffusion process is a stochastic process in which discrete movements (i.e. jumps) take place at fixed or random arrival times. Those jumps represent transitions between discrete states and the time spent on a given state is called holding time (or sojourn time). In \cite{Merton},  Merton introduced  a jump-diffusion model for pricing derivatives as follows 
\begin{align}
\frac{dS_t}{S}& = \mu \, dt + \sigma \, dW_t + dJ_t ,\nonumber \\
S&= \lim_{u \to t^{-}} S_u, \nonumber
\end{align}
where $\mu$ and $\sigma$ are constants and $J_t $ is a jump process independent of $W_t$.  Since the publication of the Merton paper \cite{Merton}, several jump-diffusion models have been considered in the academic finance literature, such as the compound Poisson model, the Kou model, the  Stochastic Volatility Jump (SVJ) model (see \cite{Kou1}, \cite{Kou} and references therein), the Bates model \cite{Bates} and the Bates Hull-White model \cite{Kienitz}.
\subsection{L\'{e}vy processes}
In this paper, we consider L\'{e}vy processes, 
commonly used in mathematical finance because they are very flexible and have a simple structure. Further, they provide the appropriate tools to adequately and consistently describe the evolution of asset returns, both in the real and in the risk-neutral world. 
A one-dimensional  L\'{e}vy process defined on $(\Omega, \cF, \{\cF\}_{t \geq 0}, \bQ)$,  is a c\`{a}dl\`{a}g, adapted process  $L= \{L_t\}_{t\geq 0}$ with $L_0=0$ a.s., having stationary (homogeneous) and independent increments and also, it is continuous in probability (see, for example, \cite{Oksendal, Oksendal1, Barndorff-Nielsen, Boyarchenko}).

When the discounted process $\{e^{(r-d)t} S_t\}_{t\ge 0}$ is a martingale under $\bQ$, the asset price dynamics under the L\'{e}vy process $L$ can be modeled as 
 \begin{align}
S_t = S_0 \exp{(-(r-d)t\;+L_t)}, \nonumber
\end{align}
where the parameter $r$ is the (domestic) risk-free interest rate and $d\ge 0$ is the continuous dividend yield of the asset. Several processes of this kind have been studied. Well-known models are the Variance Gamma (VG) \cite{Madan}, \cite{Madan1} the Normal Inverse Gaussian (NIG) \cite{Barndorff1977} and the Carr-Geman-Madan-Yor (CGMY).

\subsection{Normal Inverse Gaussian (NIG)}
The NIG  process was first introduced in 1977 \cite{Barndorff1977} and adopted in finance in 1997 \cite{Barndorff1997} as a handy model to represent fat tails and skews. 
Let us denote by $\mu\in\bR$ the location parameter, $\alpha$ the tailedness, $\beta$ the skewness satisfying $0\le |\beta|\le\alpha$, $\delta >0$ the scale, $\eta = \sqrt{\alpha^2 - \beta^2}$ and by $K_1$ a modified Bessel function of the third kind.  The NIG distribution is defined on the whole real line having  density function (\cite{Barndorff})
\[ 
f(x; \alpha,\beta,\mu,\delta) = \frac{\alpha\delta K_1 \left(\alpha\sqrt{\delta^2 + (x - \mu)^2}\right)}{\pi \sqrt{\delta^2 + (x - \mu)^2}} \; e^{\delta \eta + \beta (x - \mu)}, \quad x\in\bR.
\]
The NIG process of L\'{e}vy type can be represented via random time change of a Brownian motion as follows.

Let  $W^{\eta} = \{W^{\eta} (t)\}_{t\ge 0}$ be a Brownian motion with drift $\eta$ and diffusion coefficient 1, and 
\[ 
A_t=\inf\{s>0 : W^{\eta} (s)=\delta t\}, \quad t\ge 0,
\] 
be the inverse Gaussian process with parameters $\delta$ and $\eta$. For each $t\ge 0$, $A_t$ is a random time defined as the first passage time to level $\delta t$ of $W^{\eta}$. Further, denote by $W^{\beta} $ a second Brownian motion, stochastically independent of $W^{\eta}$, with drift $\beta$ and diffusion coefficient 1. Thus the NIG process $X= \{X_t\}_{t\ge 0}$ is the random time changed process 
\[
X_t :=W^{\beta}(A_t) + \mu t.
\]
It can be interpreted as a subordination of a Brownian motion by the inverse Gaussian process (\cite{Barndorff1997}). The distribution of the unit period increment $X_{t+1} - X_t$ follows the NIG distribution.

\subsection{Variance Gamma (VG)}
The VG model was first introduced by Madan et al. in \cite{Madan}, and then widely used to describe the behaviour of stock prices (see \cite{Madan1} and references therein). The goal of using the VG model for fitting stock prices is to improve, with respect to the classical GBM, the ability to replicate skewness and kurtosis of the return distribution. In fact, the VG process is an extension of the GBM aimed at solving some shortcomings of the Black and Scholes model.  A gamma process $\gamma(t) \equiv \gamma(t;\mu, \nu),\, t\ge 0,$ is a continuous-time stochastic process with mean rate $\mu $ and  variance rate $\nu $, such that for any $h>0$, the increments $\gamma(t+h) - \gamma(t)$ over non-overlapping intervals of equal length, are independent gamma distributed random variables with shape parameter $\alpha=\mu^2 h/\nu$ and scale parameter $\beta=\mu/\nu$.

A VG process of L\'{e}vy type,  $X= \{X_t\}_{t\ge 0}$, can be represented in two different ways:
\begin{itemize}
	\item {} As a difference $X_t := U_t - D_t$, where $U_t$ and $D_t$ are processes with i.i.d.  gamma distributed increments. Notice that, as the gamma distribution assumes only positive values, the process $X$ is increasing; 
	\item {}  As a subordination of a Brownian motion by a gamma process, i.e. $X$ is the random time changed process $X_t := W(\gamma(t))$, where $(\gamma(t))_{t\ge 0}$ is a gamma process with unit mean $\mu$ and variance rate $\nu $, and $W(t)$ is a Brownian motion with zero drift and variance $\sigma^2$. 
\end{itemize} 
In this work, we shall consider the second representation of a VG process (see Section \ref{Sec:Heston-CIR Levy model for FX market}).

\subsection{Stylized facts on returns, options' dynamics and Variance Gamma (VG)} \label{Sec:StylFacts}

As mentioned several reasons lead us to adopt a VG model. The stylized facts about returns, options' dynamics are documented by Fama \cite{Fama}, Akgiray et al. \cite{Akgiray}, Bates \cite{Bates}, Madan et al. \cite{Madan}, Campa et al. \cite{Campa}, etc. 
Namely, they are: a) ‘‘smiles’’ and ‘‘smirks’’ effects b) jumps c) finite sum of the absolute log price changes d) excess of kurtosis/long tailedness e) finite moments for at least the lower powers of returns; f) extension to multivariate processes with elliptical multivariate distributions to ensure consistency with the capital asset pricing model (CAPM).

Regarding ‘‘smiles’’ and ‘‘smirks’’, it is well known that the Black-Scholes formula is strongly biased across both moneyness and maturity and that underprices deep out-of-the-money puts and calls. For example 
Rubinstein \cite{Rubinstein1, Rubinstein2} reports evidence that implied volatilities are higher for deeply in- or out-of-the-money options. This is because stock return distributions are negatively skewed with higher kurtosis than allowable in a BS log-normal distribution \cite{Bates}. Furthermore "pure diffusion based models have  difficulties  in  explaining  smile  effects  in,  in  particular,  short-dated  option prices" \cite{Madan1}. 

The so-called stochastic volatility (SV) models were one of the first solutions to the problem as they allowed a flexible distributional structure by correlating volatility shocks and underlying stock returns. The correlation in SV models controls both skewness and kurtosis. The downside of that approach is that the volatility is a diffusion process which means it can only follow a continuous sample path (thus is unable to internalize enough short-term kurtosis \cite{Bakshi}).

To solve that concern, jump-diffusion models were introduced. In fact, they are capable to explain negative skewness and high implicit kurtosis in option prices where the random discontinuous jumps are represented by a Poisson component \cite{Bakshi}.

On the other hand, the VG process (for which the Black Scholes model is a parametric special case) is a pure jump process and there is no diffusion component. In the VG process, the returns are normally distributed, conditional on the realization of a random time with gamma density. Therefore the resulting stochastic process and associated option pricing model provide "a robust three parameter model" \cite{Madan}. Apart from the volatility of the Brownian motion the VG can control both kurtosis and skewness. 

Last but not least, concerning potential extensions to elliptical multivariate distributions consistent with the CAPM \cite{Owen}, Madan et al. \cite{Madan} have proposed a suitable generalization (even though a shortcoming of their approach is that all marginal distributions have identical kurtosis).

\subsection{Organization of the article}
This article is organized as follows. Section \ref{Sec:Heston-CIR Levy model for FX market} explains the Heston-CIR L\'{e}vy model for FX market. In Section \ref{Sec:Uniqueness} we study the local existence and uniqueness of solution for the  stochastic differential equations as described by the Variance Gamma process.
%a L\'{e}vy process.
Section \ref{Sec:Convergence of the simulation} shows that the forward Euler-Maruyama approximation method converges almost surely to the solution of the Heston-CIR model.
%the convergence for the Heston-CIR L\'{e}vy process. 
Section \ref{Sec:Dataset} illustrates the considered dataset. Section \ref{Sec:NumericalExperiments} displays numerical simulations obtained with the LSM method for L\'{e}vy processes on the American put and call options and compares estimated prices with real-wolrd market prices. Section \ref{Sec:Conclusions} concludes.

\section{The Heston-CIR of VG L\'{e}vy type model for FX market} \label{Sec:Heston-CIR Levy model for FX market}
This section describes a generalization of the Heston-CIR model \eqref{HCmodel} for evaluating options price, under the domestic risk-neutral measure $\bQ$, in FX markets. The resulting model is evaluated at random times given by a gamma process. 

Let $(\Omega,\cF,\{\cF_t\}_{t \ge 0},\bQ)$ be a complete filtered probability space, where $\bQ$ is the risk-neutral probability measure, and let us denote by $\E$ the expected value operator. On such a space, let $\bW(t)=(W_1(t),W_2(t),W_3(t),W_4(t))$ be a $4$-dimensional Brownian motion with instantaneous correlation matrix
\begin{equation*}
  \Sigma_{\bW}=\begin{pmatrix} 1 & \rho_{sv} & \rho_{sd} & \rho_{sf} \\
  \rho_{sv} & 1 & \rho_{vd} & \rho_{vf} \\
  \rho_{sd} & \rho_{vd} & 1 &\rho_{df}\\
  \rho_{sf} & \rho_{vf} & \rho_{df} & 1 \end{pmatrix}.
\end{equation*}
Clearly, being them correlation coefficients, it must hold $\rho_{sv},\rho_{sd},\rho_{sf},\rho_{vd},\rho_{vf}, \rho_{df} \in (-1,1)$ and $\Sigma_{\bW}$ has to be positive definite (assuming that $W_i$ are linearly independent).
%, by Sylvester's criterion, one must also ask for the following two additional conditions
%\begin{align*}
%    &1+2\rho_{sv}\rho_{vd}\rho_{sd}-\rho_{sd}^2-\rho_{vd}^2-\rho_{sv}^2>0\\
%    &1+2\rho_{sv}\rho_{vd}\rho_{sd}+2\rho_{vd}\rho_{vd}\rho_{df}+2\rho_{vf}\rho_{sf}\rho_{sv}-2\rho_{sf}\rho_{vd}\rho_{sd}\rho_{vf}-2\rho_{sf}\rho_{df}\rho_{sd}\rho_{vf}-2\rho_{sf}\rho_{vd}\rho_{sv}\rho_{df}\\
%    &\quad -\rho_{sv}^2-\rho_{sd}^2-\rho_{sf}^2-\rho_{vd}^2-\rho_{vf}^2-\rho_{df}^2>0
%\end{align*}
Let us now consider a Gamma subordinator $\gamma(t)$ independent of $\bW(t)$. Precisely, as introduced in the previous section, a Gamma subordinator with mean rate $\mu$ and variance rate $\nu$ is a pure-jump increasing L\'evy process such that $\gamma(0)=0$ and $\gamma(t+h)-\gamma(t)\overset{d}{=}\gamma(h)$ is a Gamma distributed random variable with shape parameter $\alpha_h=\frac{\mu^2h}{\nu}$ and scale parameter $\beta_h=\frac{\mu}{\nu}$, where with $\overset{d}{=}$ we denote the equality in distribution. One can easily express the L\'evy measure of $\gamma(t)$ in terms of the shape and scale parameters $\alpha:=\alpha_1$ and $\beta:=\beta_1$, as
\begin{equation*}
    g_\gamma(ds)=g_\gamma(s)ds=\frac{\alpha}{s}e^{-\beta s}1_{\R^+}(s)ds,
\end{equation*}
where $\R^+:=(0,+\infty)$ and for any set $A$ and any $B \subset A$ we denote for any $x \in A$
\begin{equation*}
    1_{B}(x)=\begin{cases} 1 & x \in B \\
    0 & x \not \in B. \end{cases}
\end{equation*}
Clearly, $g_\gamma(ds)$ satisfies the usual integrability assumption on the L\'evy measure of a subordinator, i.e.
\begin{equation*}
    \int_{0}^{+\infty}(1 \wedge s)g_{\gamma}(ds)<+\infty,
\end{equation*}
where for any $n \in \N$ and any real numbers $x_1,x_2,\dots,x_n \in \R$ we denote
\begin{align*}
    x_1 \wedge x_2 \wedge \dots \wedge x_n&=\min\{x_1,\dots,x_n\}\\
    x_1 \vee x_2 \vee \dots \vee x_n&=\max\{x_1,\dots,x_n\}.
\end{align*}
Let us stress that $g_\gamma(0,+\infty)=+\infty$, i.e. the subordinator has infinite activity. This guarantees that we cannot reduce $\gamma(t)$ to the (trivial) case of a compound Poisson process (see \cite[Section 1.2]{Bertoin}). Up to a time-scaling, we could always assume $\alpha=1$. However, in the following, we will, in any case, consider any shape parameter $\alpha>0$.
%Indeed, if $\gamma(t)$ admits shape and scale parameters $\alpha,\beta>0$ and $c>0$, the process $\gamma_c(t):=\gamma(ct)$ is still a Gamma subordinator (adapted with respect to the filtration $\cF_{ct}$) with $\gamma_c(t+h)-\gamma_c(t) \overset{d}{=}\gamma(ch)$ which is Gamma distributed with shape parameter $\frac{\mu ch}{\nu^2}$ and scale parameter $\frac{\mu}{\nu}$. Hence, if we set $c=\frac{\nu^2}{\mu}$, $\gamma_c(t)$ is a Gamma subordinator with shape parameter $1$..\\
Now let us consider, for any $j=1,2,3,4$, the subordinate Brownian motion $\Gamma_j(t)=W_j(\gamma(t))$. A deep study on subordinate Brownian motions is given in \cite{KimSBM}. In particular, if $\alpha=1$, $\Gamma_j(t)$ is a Variance Gamma process for any $j=1,2,3,4$, as observed in \cite{Madan1}. By virtue of the previous observation, we will call $\Gamma_j(t)$ a Variance Gamma process even if $\alpha \not = 1$, as they exhibit, up to a suitable time-scaling, the same properties. In particular, let us recall that $\Gamma_j(t)$ admits finite moments of any order and its sample paths are almost surely of bounded variation. Finally, we denote $\Gamma(t)=(\Gamma_1(t),\Gamma_2(t),\Gamma_3(t),\Gamma_4(t))$.

Fix now $t_0>0$. The Heston-CIR of VG L\'{e}vy type model is defined by the following system of SDEs, for $t \ge t_0$:
%The Heston-CIR of VG L\'{e}vy type model is defined by the following system of  SDEs, for $t\in [t_0,T]\; (0 \le t_0 < T)$
\begin{equation}\label{HCL}
\begin{aligned}\small
& dS(t) = S(t-) (r^d(t-) - r^f(t-)) \ dt+ S(t-)(\theta_s d\gamma(t)+ \sqrt{V(t-)} \,  d\Gamma_1(t),\quad S(t_0) = S_0 > 0  \\
& dV(t) = \kappa_v (a_v - V(t-)) \ dt +\theta_v d\gamma(t)+ \sigma_v \sqrt{V(t-)} \, d\Gamma_2(t),\quad V(t_0) = V_0 > 0  \\
& dr^d(t)= \kappa_d (a_d - r^d(t-)) \ dt +\theta_d d\gamma(t)+ \sigma_d \sqrt{r^d(t-)} \, d\Gamma_3(t),\quad r^d(t_0) = r^d_0 > 0   \\
& dr^f(t)= \left(  \kappa_f (a_f - r^f(t-)) -  \sigma_f\rho_{sf} \sqrt{V(t-)\, r^f(t-)}\right)  dt+\theta_f d\gamma(t)+ \sigma_f  \sqrt{r^f(t-)} \, d\Gamma_4(t),\\
&\qquad\qquad\qquad\qquad\qquad\qquad\qquad\qquad\qquad\qquad\qquad r^f(t_0) = r^f_0 > 0 ,  
\end{aligned}
\end{equation}
where $V(t)$ is the stochastic variance of the underlying asset price $S(t)$, $r^d(t)$ and $r^f(t)$ are, respectively, the domestic and foreign stochastic short interest rates. The parameters $\kappa_v,\kappa_d,\kappa_f>0$ are the speeds of mean reversion, $a_v,a_d,a_f>0$ are the foreign long-run means, $\sigma_v,\sigma_d,\sigma_f>0$ are the volatilities, $\theta_v,\theta_d,\theta_f \ge 0$ are the drifts respectively of $V,r^d$ and $r^f$. Moreover, we assume that $\rho_{sf}$ is the correlation between a domestic asset $S$ and the foreign interest rate $r^f$ and $\theta_s \ge 0$ is its drift. For any function $f:\R_{t_0}^+ \to \R$, where $\R_{t_0}^+:=[t_0,+\infty)$, we denote, if it exists, $f(t-)=\lim_{s \to t^-}f(s)$. Finally, we assume that the initial data $S_0,V_0,r_0^d,r_0^f>0$ are deterministic.
%The gamma process $(\gamma(t))_{t\in [t_0,T]}$ has been defined in the previous section. The L\'{e}vy representation of the distribution of increments per unit time has L\'{e}vy measure explicitly given by 
%
%\begin{equation}\label{Levymeasure}
%g_{\gamma} (x) dx =\frac{\alpha}{x} \, e^{-\beta x}\, dx,\;\mbox{for},\; %\mbox{for $x > 0$ and 0 otherwise}.
%\end{equation}
%with $\alpha=\mu^2/\nu$ and  $\beta=\mu/\nu$.

%We define the multidimensional time changed  Brownian motion ${\bf W}(\gamma(t)) := [ W_1(\gamma(t)), W_2(\gamma(t)), W_3(\gamma(t)), W_4(\gamma(t))]^T$, with ${\bf W}(t_0)=0$ and correlation matrix given by
%\begin{align}\label{cormatrix}
%&  d{\bf W}(\gamma(t)) \; d{\bf W}(\gamma(t))^T \equiv  \Sigma_{\bf W} d \gamma(t) = \begin{pmatrix} 1 & \rho_{sv} & \rho_{sd} & \rho_{sf}  \\  \rho_{sv} & 1  & \rho_{vd} &\rho_{vf} \\ \rho_{sd} & \rho_{vd} & 1 & \rho_{df} \\ \rho_{sf} & \rho_{vf} & \rho_{df} & 1     \end{pmatrix} d \gamma(t).
%\end{align}
%
In the next section, we will investigate (local) the existence and uniqueness of the solution of the system defining the Heston-CIR VG L\'evy type model.

\section{Existence and uniqueness of the solution of VG L\'{e}vy type model} \label{Sec:Uniqueness}

First of all, let us consider, without loss of generality, $t_0=0$. Let us fix some further notation. For any topological space $\mathcal{X}$ we denote by $\mathcal{B}(\mathcal{X})$ its Borel $\sigma$-algebra. Moreover, for any c\'adl\'ag function $f:\R^+_0 \to \R$ we denote $\Delta f(t)=f(t)-f(t-)$.  For any $n \in \N$ and $x \in \R^n$ we denote $x=(x_1,\dots,x_n)$ to identify its components and the action of a function $g:\R^n \to \R$ will be denoted both with $g(x)$ or $g(x_1,\dots,x_n)$ depending on the necessity of highlighting the single components of the argument. For $x,y \in \R^n$ we denote $\langle x,y\rangle=\sum_{j=1}^{n}x_jy_j$ the usual scalar product and $|x|^2=\sum_{j=1}^{n}x_j^2$ the Euclidean norm. Moreover, for any fixed $n,m \in \N$ and any bounded function $\phi:\R^n \to \R^m$ we denote $\left\|\phi \right\|_{\infty}=\max_{x \in \R^n}|\phi(x)|$. For any function $\phi:\R^n \to \R^m$ we denote the support of $\phi$ as ${\rm supp}(\phi):={\rm cl}(\{x \in \R^n: \ \phi(x)=0\})$, where, for any $A \subseteq \R^n$, ${\rm cl}(A)$ is the topological closure of $A$. We denote by $C_b(\R^n)$ the space of continuous and bounded functions $f:\R^n \to \R$ and by $C_c^\infty(\R^n)$ the space of infinitely differentiable functions $\phi:\R^n \to \R$ with compact support. For any $m \in \N$ we denote by $C^m(\R^n)$ the space of continuous functions $f:\R^n \to \R$ with continuous partial derivatives up to order $m$ and by $C^m_b(\R^n)$ its subspace of bounded function with bounded partial derivatives. Finally, given a sequence $\{a_k\}_{k \in \N}$, we denote $a_k \uparrow a$ (resp. $a_k \downarrow a$) as $k \to +\infty$ if $a_k$ is non-decreasing (resp. non-increasing) and $\lim_{k \to +\infty}a_k=a$. In the proofs, we will denote by $C>0$ any generic constant whose value is not crucial.

To prove the local existence and uniqueness of the solutions of \eqref{HCL} we will make use of several different techniques from both the theory of stochastic differential equations with jumps and c\'adl\'ag rough differential equations. In place of proving a single existence and uniqueness theorem for the whole system, we will proceed step by step, providing first the results concerning the equations of $V(t)$ and $r^d(t)$, then the one of $r^f(t)$ and finally the one of $S(t)$. Thus, we first want to focus on the equations
\begin{align}
    dV(t) = \kappa_v (a_v - V(t-)) \ dt +\theta_v d\gamma(t)+ \sigma_v \sqrt{V(t-)} \, d\Gamma_2(t),\quad V(0) = V_0 > 0,  \label{eq:V0}\\
dr^d(t)= \kappa_d (a_d - r^d(t-)) \ dt +\theta_d d\gamma(t)+ \sigma_d \sqrt{r^d(t-)} \, d\Gamma_3(t),\quad r^d(0) = r^d_0 > 0.   \label{eq:rd0}
\end{align}
To do this, we need the following easy (but technical) Lemma which is demonstrated in Appendix \ref{proof:Levy}.

%%%%%%%%%%%%%%%%
%%%LEMMA LEVY%%%
\begin{lemma}\label{lem:Levy}
Let
\begin{equation*}
    p_1(s,u)=\frac{1}{\sqrt{2\pi s}}e^{-\frac{u^2}{2s}}1_{(0,+\infty)}(s), \ (s,u) \in \R^2,
\end{equation*}
For any $j=1,2,3,4$, the process $(\gamma(t),\Gamma_j(t))$ is a pure jump L\'evy process on $\R^2$ with L\'evy measure
\begin{equation}
\nu_1(dsdu)=\nu_1(s,u)dsdu=g_\gamma(s)p_1(s,u)dsdu.
\end{equation}
\end{lemma}
%%%%%%%%%%%%%%%%
%%%%%%%%%%%%%%%

With this property in mind, we can prove the following existence and uniqueness theorem by combining the arguments of \cite[Theorems 2.4 and 2.8]{Xi} and \cite[Theorem 3.1 and 3.2]{Fu}.

%%%%%%%% Theorem exunV0rd0 %%%%%%%
\begin{theorem} \label{thm:exunV0rd0}
Equations \eqref{eq:V0} and \eqref{eq:rd0} admit pathwise unique non-negative strong solutions respectively up to the Markov (or stopping) times $\tau^V_0$ and $\tau^d_0$.
\end{theorem}

The proof of Theorem \ref{thm:exunV0rd0} is provided in Appendix \ref{proof:exunV0rd0}.
%%%%%%%%%%%

Now let us handle the Equation
\begin{align}\label{eq:rf0}
\begin{split}
& dr^f(t)= \left(  \kappa_f (a_f - r^f(t-)) -  \sigma_f\rho_{sf} \sqrt{V(t-)\, r^f(t-)}\right)  dt+\theta_f d\gamma(t)+ \sigma_f  \sqrt{r^f(t-)} \, d\Gamma_4(t),\\
&\qquad\qquad\qquad\qquad\qquad\qquad\qquad\qquad\qquad\qquad\qquad r^f(t_0) = r^f_0 > 0 ,  
\end{split}
\end{align}
It will be convenient to consider it as coupled with \eqref{eq:V0}. However, the processes $\Gamma_2(t)$ and $\Gamma_4(t)$ that constitute part of the noise are correlated, not only since they are obtained from their respective parent processes $W_2(t)$ and $W_4(t)$ by means of the same time-change $\gamma(t)$, but also since the parent processes themselves are correlated, with instantaneous correlation matrix
\begin{equation*}
    \Sigma_{\bW}^{(2,4)}=\begin{pmatrix} 1 & \rho_{vf} \\ \rho_{vf} & 1 \end{pmatrix}.
\end{equation*}
Let us consider a standard $2$-dimensional Brownian motion $\bwW(t)=(\wW_1(t),\wW_2(t))$ and let $R$ be the Cholesky factor (see \cite[Corollary 7.2.9]{Horn}). In particular, in this case, one can evaluate $R$ explicitly as
\begin{equation*}
    R=\begin{pmatrix} 1 & 0 \\
    \rho_{vf} & \sqrt{1-\rho^2_{vf}}
    \end{pmatrix}.
\end{equation*}
It is well known (see, for instance, \cite[Section XI.2]{Asmussen}) that, denoting $\bW_{(2,4)}(t)=(W_2(t),W_4(t))$, it holds $\bW_{(2,4)}^T(t)\overset{d}{=} R\bwW^T(t)$, where with $M^T$ we mean the transposed of the matrix $M$. This means that $R^{-1}\bW_{(2,4)}^T(t)$ is a standard $2$-dimensional Brownian motion ad we can directly set $\bwW^T(t)=R^{-1}\bW_{(2,4)}^T(t)$ without loss of generality. Defining $\wGamma(t)=\bwW(\gamma(t))$, we have $\Gamma_{(2,4)}(t):=(\Gamma_2(t),\Gamma_4(t))=R\wGamma(t)$. With this in mind, we can rewrite Equations \eqref{eq:V0} and \eqref{eq:rf0} as
\begin{align}
    dV(t)&=\kappa_v(a_v-V(t-))dt+\theta_v d\gamma(t)+\sigma_v \sqrt{V(t-)}d\wGamma_1(t), & V_0>0 \label{eq:V0R}\\
    \begin{split}
    dr^f(t)&=(\kappa_f(a_f-r^f(t-))-\sigma_f \rho_{sf} \sqrt{V(t-)r^f(t-)})dt+\theta_f d\gamma(t)\\&+\sigma_f \rho_{vf}\sqrt{r^f(t-)}d\wGamma_1(t)+\sigma_f \sqrt{1-\rho^2_{vf}}\sqrt{r^f(t-)}d\wGamma_2(t),
    \end{split} & r^f_0>0. \label{eq:rf0R}
\end{align}

This time we need a slightly different technical Lemma, that is proved in the same way as Lemma \ref{lem:Levy}.
\begin{lemma}\label{lem:Levy2}
Let
\begin{align*}
    p_2(s,u)&=\frac{1}{2\pi s}e^{-\frac{|u|^2}{2s}}1_{(0,+\infty)}(s), \  s \in \R, \ u \in \R^2,\\
\end{align*}
The process $(\gamma(t),\wGamma^{(1)}(t),\wGamma^{(2)}(t))$ is a pure jump L\'evy process on $\R^3$ with L\'evy measure
\begin{equation}
\nu_2(dsdu)=\nu_2(s,u)dsdu=g_\gamma(s)p_2(s,u)dsdu.
\end{equation}
\end{lemma}
%This is proved in the exact same way as Lemma \ref{lem:Levy}.

Now we are ready to state the following existence and uniqueness theorem.
\begin{theorem}\label{thm:exunrf0}
Equation \eqref{eq:rf0} admits a pathwise unique non-negative strong solution up to a Markov time $\tau_0^f$.
\end{theorem}

The proof of Theorem \ref{thm:exunrf0} is provided in Appendix \ref{proof:exunrf0}.
%%%%%%%%%%%

\begin{remark}
It is clear that the localization argument adopted in the case $\rho_{sf}<0$ is also valid in the case $\rho_{sf}\ge 0$ and for Equations \eqref{eq:V0} and \eqref{eq:rd0}.
Let us define a somewhat natural extension of the processes $V(t)$, $r^d(t)$ and $r^f(t)$ when they turn negative.
Up to now, we were not able to prove that a similar extension holds as $\rho_{sf}<0$.\\
Let us also recall that, since we are using \cite[Theorem 2.2]{Stroock} or \cite[Theorem 6.2.3]{Applebaum}, we have the processes $V(t)$, $r^d(t)$ and $r^f(t)$ are almost surely c\'adl\'ag. Without loss of generality, we can assume that for any $\omega \in \Omega$ the paths $V(\cdot,\omega)$, $r^d(\cdot,\omega)$ and $r^f(\cdot,\omega)$ are c\'adl\'ag respectively up to $\tau^V_0$, $\tau^d_0$ and $\tau^f_0$.
\end{remark}
Now we are finally ready to prove that Equation
\begin{align} \label{eq:S0}
    dS(t)&=S(t-)(r^d(t-)-r^f(t-))dt+S(t-)(\theta_sd\gamma(t)+\sqrt{V(t-)}d\Gamma_1(t)), \\
    & \text{with} \quad S(0)=S_0>0 \nonumber
\end{align}
admits a pathwise unique strong solution up to a Markov time.
\begin{theorem}
Equation \eqref{eq:S0} admits a pathwise unique non-negative c\'adl\'ag strong solution up to a Markov time $\tau_0^S$.
\end{theorem}
\begin{proof}
Let $\tau_0^V,\tau_0^d$ and $\tau_0^f$ be the Markov times up to which Equations \eqref{eq:V0}, \eqref{eq:rd0} and \eqref{eq:rf0} admit a pathwise unique strong solution and let $\tau_0^S=\tau_0^V \wedge \tau_0^d \wedge \tau_0^f$. Without loss of generality, we can assume that $\gamma(t)$ and $\Gamma_1(t)$ are of bounded variation and c\'adl\'ag for any $\omega \in \Omega$, while $V(t)$, $r^d(t)$ and $r^f(t)$ are c\'adl\'ag for any $\omega \in \Omega$. Then we define the process, for $t<\tau_0^S$,
\begin{equation*}
    X(t)=\int_0^t(r^d(\tau-)-r^f(\tau-))d\tau+\theta_s\gamma(t)+\int_0^t \sqrt{V(\tau-)}d\Gamma_1(\tau),
\end{equation*}
where we can consider each integral as a Lebesgue-Stieltjes integral for any fixed $\omega \in \Omega$.\\
Now fix $\omega \in \Omega$ and any $T<\tau_0^S(\omega)$. Then we can consider the random rough differential equation
\begin{equation*}
    dS(t,\omega)=S(t-,\omega)dX(t,\omega), \qquad t \in [0,T], \qquad S(0,\omega)=S_0>0.
\end{equation*}
Such an equation admits a unique c\'adl\'ag solution for each $\omega$, thanks to \cite[Theorem 1.13]{Friz}, while \cite[Proposition 6.9]{Friz} guarantees that $\omega \mapsto S(\cdot,\omega)$ is the unique strong solution of Equation \eqref{eq:S0}, concluding the proof.
\end{proof}
\begin{remark}
By definition, $\tau_0^S \le \tau_0^V,\tau_0^d,\tau_0^f$. Hence we can define $\tau_0^S$ as the local existence time threshold for the whole Heston-CIR system \eqref{HCL}. 
\end{remark}
In the next section we will focus on the convergence of the forward Euler-Maruyama scheme to the solution of our Heston-CIR model. To do this, we will use a localization argument which is similar to the one adopted in the proof of Theorem \ref{thm:exunrf0} in the case $\rho_{sf}<0$.
\section{Convergence of the Euler discretization method to the Heston-CIR of VG  L\'{e}vy type model} \label{Sec:Convergence of the simulation}

As we stated in the previous Section, we want to prove some form of a convergence of the forward Euler-Maruyama scheme for the Heston-CIR system \eqref{HCL}. As we did for the proof of Theorem \ref{thm:exunrf0}, we first need to recast the equations in order to handle the correlation structure between the processes $\Gamma_j(t)$, $j=1,2,3,4$. Arguing as before, we can consider the Cholesky factorization $\Sigma_{\bW}=HH^T$ (where $H=(H_{i,j})_{i,j=1,\dots,4}$) and define a $4$-dimensional standard Brownian motion $\overline{\bW}(t)=(\oW_1(t),\oW_2(t),\oW_3(t),\oW_4(t))$ by setting
\begin{equation*}
    \boW^T(t)=R^{-1}\bW^T(t).
\end{equation*}
Then we can define the process $\oGamma(t):=(\oGamma_1(t),\oGamma_2(t),\oGamma_3(t),\oGamma_4(t))$ as $\oGamma(t)=\boW(\gamma(t))$. Again, to find the L\'evy-It\^o decomposition of $(\gamma(t),\oGamma(t))$, we need the following technical Lemma, whose proof is identical to the one of Lemma \ref{lem:Levy}.
\begin{lemma}\label{lem:Levy3}
Let
\begin{equation*}
    p_4(s,u)=\frac{1}{4\pi s^2}e^{-\frac{|u|^2}{2s}}1_{\R^+}(s), \ s \in \R, \ u \in \R^4.
\end{equation*}
Then the process $(\gamma(t),\oGamma(t))$ is a pure jump L\'evy process on $\R^5$ with L\'evy measure
\begin{equation*}
    \nu_4(dsdu)=\nu_4(s,u)dsdu=g_\gamma(s)p_4(s,u)dsdu.
\end{equation*}
\end{lemma}
We can deduce from the previous Lemma the L\'evy-It\^o decomposition of $(\gamma(t),\oGamma(t))$, by defining the Poisson random measure, for $t \in \R^+$ and $A \in \cB(\R^5 \setminus \{0\})$ as
\begin{equation*}
N_{\oGamma}(t,A)=\sum_{0 \le s \le t}1_{A}((\Delta \gamma(s), \Delta \oGamma(s))),
\end{equation*}
and the compensated Poisson measure as $\wN_{\oGamma}(dt,dsdu)=N_{\oGamma}(dt,dsdu)-\nu_4(dsdu)dt$ and then observing that
\begin{equation*}
    (\gamma(t),\oGamma(t))=\int_0^t\int_{\R^5}(s,u)\wN_{\oGamma}(dt,dsdu).
\end{equation*}
Once we have obtained the L\'evy-It\^o decomposition of $(\gamma(t),\oGamma(t))$, we can rewrite the Heston-CIR system as follows
\begin{equation*}
    dX(t)=b(X(t-))dt+\int_{\R^5}g(s,u,X(t-))\wN_{\oGamma}(dt,dsdu),
\end{equation*}
where $X(t)=(S(t),V(t),r^d(t),r^f(t))^T$,
\begin{equation*}
    b(x)=\begin{pmatrix} x_1(x_3-x_4) \\ \kappa_v(a_v-x_2) \\ \kappa_d(a_d-x_3) \\ \kappa_f(a_f-x_4)-\sigma_f \rho_{sf}\sqrt{x_2x_4}
    \end{pmatrix}
\end{equation*}
and
\begin{equation*}
    g(s,u,x)=\begin{pmatrix} x_1(\theta_ss+H_{1,1}u_1\sqrt{x_2}) \\
    \theta_vs+(H_{2,1}u_1+H_{2,2}u_2)\sigma_v\sqrt{x_2}\\
    \theta_ds+(H_{3,1}u_1+H_{3,2}u_2+H_{3,3}u_3)\sigma_d\sqrt{x_3}\\
    \theta_fs+(H_{4,1}u_1+H_{4,2}u_2+H_{4,3}u_3+R_{4,4}u_4)\sigma_f\sqrt{x_4}.
    \end{pmatrix}
\end{equation*}
Let us also recall that recursive formulae for $H_{i,j}$ are known:
\begin{align}\label{eq:Choleskyfactor}
\begin{split}
H_{1,1}&=1\\
H_{i,j}&=\frac{1}{H_{j,j}}\left((\Sigma_{\bW})_{i,j}-\sum_{k=1}^{j-1}H_{i,k}\right) , \qquad i<j\\
H_{j,j}&=\sqrt{1-\sum_{k=1}^{j-1}H_{j,k}^2}.
\end{split}
\end{align}
Once this is done, we proceed with the localization of the equation. Precisely, let $n_0 \in \N$ be such that $(S_0,V_0,r^d_0,r^f_0) \in \left[\frac{1}{n_0},n_0\right]$ and consider $\psi^{(3)}_n$ for $n \ge n_0$ as defined in the proof of Theorem \ref{thm:exunrf0}. Let then $\psi^{(5)}_n(x)=\prod_{j=1}^{4}\psi^{(3)}_n(x_j)$ for any $x \in \R^4$ and
\begin{align*}
    b_n(x)=\psi^{(5)}_n(x)b(x) \qquad g_n(s,u,x)=\psi^{(5)}_n(x)g(s,u,x).
\end{align*}
Clearly, by definition, $b_n$ is Lipscthiz and bounded. Concerning $g_n$, we need two technical observations.

\begin{lemma}\label{lem:moments}
For any $p \ge 1$ and $j=1,2,3,4$ it holds
\begin{align*}
    \int_{\R^5}|s|^p\nu_4(dsdu)<+\infty && \int_{\R^5}|u_j|^p\nu_4(dsdu)<+\infty.
\end{align*}
\end{lemma}
\begin{proof}
Fix $p \ge 1$ and let us prove the first inequality. Indeed, we have
\begin{equation*}
\int_{\R^5}|s|^p\nu_4(dsdu)=\alpha\int_0^{+\infty}s^{p-1}e^{-\beta s}\left(\int_{\R^4}p_4(s,u)du\right)ds<+\infty
\end{equation*}
since $p \ge 1$. 
To prove the second inequality, just observe that $p_4(s,u)=\prod_{j=1}^{4}p_1(s,u_j)$, so that, since $p \ge 1$, by the well-known formula for the absolute moments of the Gaussian distribution, it holds
%\begin{align*}
%\int_{\R^5}|u_j|^p\nu_4(dsdu)&=\alpha\int_0^{+\infty}s^{-1}e^{-\beta s}\prod_{\substack{i \le 4 \\ i \not = j}}\left(\int_{\R}p_1(s,u_i)du_i\right)\left(\int_{\R}|u_j|^pp_1(s,u_j)du_j\right)ds\\&=C_p\alpha\int_0^{+\infty}s^{\frac{p}{2}-1}e^{-\beta s}<+\infty.
%\end{align*}
\begin{align*}
\int_{\R^5}|u_j|^p\nu_4(dsdu)=C_p\alpha\int_0^{+\infty}s^{\frac{p}{2}-1}e^{-\beta s}<+\infty.
\end{align*}
\end{proof}
\begin{lemma}\label{lem:Lip}
For any $n \in \N$ there exist two functions $M_n,L_n:\R^5 \to \R$ such that:
\begin{itemize}
    \item[(i)] For any $s \in \R$ and $x,u \in \R^4$ it holds
    \begin{equation}
    |g_n(s,u,x)|\le M_n(s,u)
\end{equation}
\item[(ii)] For any $s \in \R$ and $x,y,u \in \R^4$ it holds
\begin{equation}
    |g_n(s,u,x)-g_n(s,u,y)|\le L_n(s,u)|x-y|
\end{equation}
\item[(iii)] For any $p \ge 1$ it holds
\begin{align*}
    \int_{\R^5}|M_n(s,u)|^p\nu_4(dsdu)<+\infty && \int_{\R^5}|L_n(s,u)|^p\nu_4(dsdu)<+\infty. 
\end{align*}
\end{itemize}
\end{lemma}
\begin{proof}
To prove (i), let us first observe that $g_n(s,u,x)=0$ whenever $x_j \not \in \left[\frac{1}{n+1},n+1\right]$ for some $j=1,2,3,4$. Hence we only need to work with $x_j \in \left[\frac{1}{n+1},n+1\right]$ for all $j=1,2,3,4$. 

Recalling that $|\psi_n^{(5)}(x)|\le 1$ for any $x \in \R^4$, we have
\begin{align*}
    |g_n(s,u,x)|\le |g(s,u,x)|\le C_n(|s|+|u_1|+|u_2|+|u_3|+|u_4|)=:M_n(s,u).
\end{align*}
To prove (ii), let us observe that for fixed $s \in \R$ and $u \in \R^4$ $g_n(s,u,\cdot)$ is a Lipschitz function, hence we only need to determine an upper bound for the gradient of $g_n(s,u,\cdot)$. Let us also stress that $g_n(s,u,\cdot)=0$ as $x_j \not \in \left[\frac{1}{n+1},n+1\right]$ for some $j=1,2,3,4$, hence we only need to consider the case $x_j \not \in \left[\frac{1}{n+1},n+1\right]$ for all $j=1,2,3,4$. By simple but cumbersome calculations, it can be shown that
\begin{equation*}
    |\nabla g_n(s,u,x)|\le C_n(|s|+|u_1|+|u_2|+|u_3|+|u_4|)=:L_n(s,u).
\end{equation*}
Finally, item (iii) follows directly from Lemma \ref{lem:moments}.
\end{proof}
Now, we consider the solution of the localized equation
\begin{equation}\label{eq:local}
    dX^n(t)=b_n(X^n(t-))dt+\int_{\R^5}g_n(s,u,X^n(t-))\wN_{\oGamma}(dt,dsdu), \ X^n(0)=X_0
\end{equation}
which exists for any $t \ge 0$ since we are under the hypotheses of \cite[Theorem 6.2.3]{Applebaum}. Moreover, if we set 
\begin{equation*}
\tau_n^X:=\inf\left\{t \ge 0: X^n(t) \not \in \left[\frac{1}{n},n\right]\right\},
\end{equation*}
by pathwise uniqueness we know that $X^n(t)=X(t)$ for any $t<\tau_n^X$. We also know that such solutions are c\'adl\'ag. However, we first want to prove an important property of $X^n(t)$ which will turn out to be useful in the following.
\begin{theorem}\label{thm:stoccont}
The process $X^n(t)$ is stochastically continuous, i.e. for any fixed $t>0$ it holds $\bQ(X^n(t-)=X^n(t))=1$.
\end{theorem}
\begin{proof}
Let us fix $t>0$ and let $h \in (0,1)$ be small enough to have $t-h>0$. By definition, it holds
\begin{equation*}
    X^n(t)=X^n(t-h)+\int_{t-h}^t b_n(X^n(z-))dz+\int_{t-h}^t\int_{\R^5}g_n(s,u,X^n(z-))\wN_{\oGamma}(dz,dsdu).
\end{equation*}
Hence, by the fact that $b_n$ is bounded and applying Kunita's first inequality (see \cite[Theorem 4.4.23]{Applebaum}) and items (i) and (iii) of Lemma \ref{lem:Lip}, we obtain
\begin{align}\label{eq:estcont1}
\begin{split}
    \E&\left[|X^n(t)-X^n(t-h)|^2\right]\le Ch^2\\
    &\quad +2\E\left[\sup_{v \in [t-h,t]}\left|\int_{t-h}^v\int_{\R^5}g_n(s,u,X^n(z-))\wN_{\oGamma}(dz,dsdu)\right|^2\right] \le Ch,
\end{split}
\end{align}
which concludes the proof by a simple application of Fatou's lemma.
\end{proof}
\begin{remark}
Let us stress that, since for any $n \ge n_0$, up to $\tau^X_n$ the processes $X^n(t)$ and $X(t)$ coincide, this also proves the stochastic continuity of the process $X(t)$ up to the Markov time $\widetilde{\tau}^X=\lim_{n \to +\infty}\tau_n^X$. By definition, it is also clear that 
\begin{equation*}
\widetilde{\tau}^X:=\inf\{t>0: \ S(t) \vee V(t) \vee r^d(t) \vee r^f(t) \le 0\}.
\end{equation*}
Hence we can conclude that $\widetilde{\tau}^X \le \tau^S$, while we are not able, up to now, to prove equality. This justifies the fact that we did not use this localization technique to prove local existence for \eqref{eq:V0}, \eqref{eq:rd0} and \eqref{eq:rf0} if $\rho_{sf} \ge 0$. Indeed, our proof guarantees a local existence time that could be possibly greater than the one obtained by localization.
\end{remark}
Now let us introduce the forward Euler-Maruyama approximation of the localized equation \eqref{eq:local}. Let us consider a stepsize $\Delta t$ and let $t_k=k\Delta t$ for $k \in \N$. Let us also denote $\Delta \gamma_k=\gamma(t_{k+1})-\gamma(t_k)$ and $\Delta \oGamma_k=\oGamma(t_{k+1})-\oGamma(t_k)$ for $k \in \N$. Then we define $X^n_{\Delta t,0}=X_0$ and, for $k \ge 0$,
\begin{align*}
X^n_{\Delta t,k+1}&=X^n_{\Delta t,k}+b_n(X^n_{\Delta t,k})\Delta t+\langle g(1,1,X^n_{\Delta t,k}),(\Delta \gamma_k,\Delta \oGamma_k)\rangle\\
&=X^n_{\Delta t,k}+\int_{k\Delta t}^{(k+1)\Delta t}b_n(X^n_{\Delta t,k})dt+\int_{k\Delta t}^{(k+1)\Delta t}\int_{\R^5}g_n(s,u,X^n_{\Delta t,k})\wN_{\oGamma}(dt,dsdu).
\end{align*}
It is clear, by induction, that
\begin{align*}
X^n_{\Delta t,k}&=X_0+\int_{0}^{k\Delta t}b_n(X^n_{\Delta t,k-1})dt+\int_{0}^{k\Delta t}\int_{\R^5}g_n(s,u,X^n_{\Delta t,k-1})\wN_{\oGamma}(dt,dsdu).
\end{align*}
We can construct a polygonal-like continuous extension of the discrete process $(X^n_{\Delta t,k})_{k \ge 0}$ as follows:
\begin{align*}
X^n_{\Delta t}(t)&=X_0+\int_{0}^{t}b_n(X^n_{\Delta t}(\eta_{\Delta t}(t)-)dt+\int_{0}^{t}\int_{\R^5}g_n(s,u,X^n_{\Delta t}(\eta_{\Delta t}(t)-)\wN_{\oGamma}(dt,dsdu),
\end{align*}
where $\eta_{\Delta t}(t)=k\Delta t$ if $k \Delta t<t\le (k+1)\Delta t$ for some non-negative integer $k \ge 0$. With the exact same arguments as in Theorem \ref{thm:stoccont} we can prove the following Lemma.
\begin{lemma}
For any $\Delta t>0$, the process $X^n_{\Delta t}(t)$ is stochastically continuous.
\end{lemma}

Now, for $m \in \N$, let us define $X^n_m(t):=X^n_{\Delta t_m}$ where $\Delta t_m=\frac{\delta t}{m^2}$ for some constant $\delta t>0$. It will be clear, in the following, that all the arguments are independent of the choice of $\delta t$, hence we can set $\delta t=1$ for simplicity. We want to prove the following theorem, along the lines of \cite[Theorem 2.3]{Gyongy}.

%%%%%%%%%% Theorem QuasiGyongy %%%%%%%%%
\begin{theorem} \label{thm:QuasiGyongy}
For any $\theta<\frac{1}{4}$ and $T>0$ there exists a random variable $\zeta_{\theta,T}>0$ such that $\bQ(\zeta_{\theta,T}<\infty)=1$ and
\begin{equation}\label{eq:supremum}
    \sup_{t \in [0,T]}|X^n(t)-X^n_m(t)|\le \zeta_{\theta,T}m^{-\theta}, \ \forall m \in \N
\end{equation}
almost surely.
\end{theorem}

%%%%%%%%%%%%%%
The proof of Theorem \ref{thm:QuasiGyongy} is provided in Appendix \ref{proof:QuasiGyongy}.
%%%%%%%%%%%

Now, we are ready to prove the convergence of the forward Euler-Maruyama scheme for our Heston-CIR model.
\begin{theorem}\label{thm:EMapprox}
Let
\begin{equation}\label{eq:EMapprox}
    X_m(t):=X_0+\int_0^t b_n(X_m(\eta_m(z)-))dz+\int_0^t \int_{\R^5}g_n(s,u,X_m(\eta_m(z)-))\wN_{\oGamma}(dt,dsdu).
\end{equation}
Then $X_m(t)$ is well defined up to a Markov time $\tau_m$ such that $\bQ(\tau_m>0)=1$. Moreover, for almost any $\omega \in \Omega$ and $T<\widetilde{\tau}^X(\omega)$ there exists $m_0 \in \N$ such that $\tau_m(\omega) > T$ for any $m \ge m_0$ and
\begin{equation}\label{eq:limitEM}
    \lim_{m \to +\infty}\sup_{t \in [0,T]}|X(t,\omega)-X_m(t,\omega)|=0.
\end{equation}
\end{theorem}
\begin{proof}
For any $n \in \N$ with $n \ge n_0$ and any $m \in \N$, let $$\tau_m^n:=\inf\left\{t>0: \ X_m(t) \not \in \left[\frac{1}{n},n\right]\right\}.$$
Then, by definition of $b_n$ and $g_n$, $X_m(t)=X_m^n(t)$ as $t<\tau_m^n$. Thus we can define $\tau_m=\lim_{n \to +\infty}\tau_m^n$. Now let \begin{align*}
    \Omega_0:=\{\omega \in \Omega: \mbox{ Inequality }&\mbox{\eqref{eq:supremum} holds for any rational $T$ and any $n \ge n_0$}\\&\mbox{and $\widetilde{\tau}^X(\omega)=\lim_{n}\tau^X_n(\omega)$}\}
\end{align*}
and observe that, clearly, $\bQ(\Omega \setminus \Omega_0)=0$. Fix $\omega \in \Omega_0$ and $T<\widetilde{\tau}^X(\omega)$. Then there exists a rational $T^\prime$ such that $T<T^\prime<\widetilde{\tau}^X(\omega)$. However, $\tau_n^X\uparrow \widetilde{\tau}^X$, hence there exists $N$ such that $T^\prime<\tau_N^X(\omega)<\widetilde{\tau}^X(\omega)$. By Inequality \eqref{eq:supremum} we have
\begin{equation}\label{eq:limit}
    \lim_{m \to +\infty}\sup_{t \in [0,T^\prime]}|X^{N+1}(t,\omega)-X^{N+1}_m(t,\omega)|=0.
\end{equation}
In particular, this means that for any $\varepsilon>0$ there exists $m_0$ such that for any $m \ge m_0$
\begin{equation}\label{eq:lesseps}
    \sup_{t \in [0,T^\prime]}|X^{N+1}(t,\omega)-X^{N+1}_m(t,\omega)|< \varepsilon.
\end{equation}
However, since $T^\prime<\tau_N^X(\omega)$, we know that $X^{N+1}(t,\omega) \in \left[\frac{1}{N},N\right]$ for any $t \in [0,T^\prime]$ and thus we can choose $\varepsilon>0$ so small that Inequality \eqref{eq:lesseps} implies $X^{N+1}_m(t,\omega) \in \left[\frac{1}{N+1},N+1\right]$ for any $t \in [0,T^\prime]$. This clearly implies that for any $m \ge m_0$ it holds $\tau_m(\omega) \ge T^\prime>T$. Finally, \eqref{eq:limitEM} is implied by \eqref{eq:limit} once we observe that for $t \in [0,T]$ it holds $X^{N+1}(t,\omega)=X(t,\omega)$ and $X^{N+1}_m(t,\omega)=X_m(t,\omega)$.
\end{proof}

\subsection {Simulation} \label{Sec:Simulation}
At this point, we use Theorem \ref{thm:EMapprox} to provide a simulation algorithm for the Heston-CIR model \eqref{HCL} on a grid of equidistant time points $0 \le t_0<t_1<\dots<t_N=T$ with $N \in \N$, assuming we assign the value $X(t_0)=X_0$. Precisely, let us set $\Delta t=\frac{T-t_0}{N}$ so that $t_j=t_0+j\Delta t$ for $0 \le j \le N$. Then we define $X_{\Delta t,j}$ by setting $X_{\Delta t,0}=X_0$ and then, for $0 \le j \le N-1$,
\begin{equation}\label{eq:recursEM}
X_{\Delta t,j+1} =X_{\Delta t,j}+b(X_{\Delta t,j})\Delta t+\langle g(1,1,X_{\Delta t,j}),(\Delta \gamma_j,\Delta \oGamma_j)\rangle.
\end{equation}
Let us stress that such a discrete-time process can be extended to a continuous time one as in equation \eqref{eq:EMapprox}, thus Theorem \ref{thm:EMapprox} guarantees that for $\Delta t$ small enough $(X_{\Delta t, j})_{0 \le j \le N}$ is almost surely a good approximation of the values of the process $X(t)$ in the nodes $t_0,\dots,t_{N-1},T$. Let us stress that for $0 \le i,j \le N-1$ with $i \not = j$, the quantities $\Delta \gamma_i$ and $\Delta \gamma_j$ are independent of each other and the same occurs for $\Delta \oGamma_i$ and $\Delta \oGamma_j$. To simulate $\Delta \gamma_j$, we just observe that these are Gamma-distributed random variables with scale parameter $\beta$ and shape parameter $\alpha \Delta t$. To simulate $\Delta \oGamma_j$, we observe that, thanks to the independence of the processes $\gamma(t)$ and $\boW(t)$, if we consider $Z=(Z_1,Z_2,Z_3,Z_4)$, where the $Z_i$'s are independent standard Gaussian random variables for $i = 1,...,4$, then $\Delta \oGamma_j \overset{d}{=} \sqrt{\Delta \gamma_j} Z$.\\
We can rewrite the recursive relation \eqref{eq:recursEM} in terms of the four components $S_{\Delta t}$, $V_{\Delta t}$, $r^d_{\Delta t}$ and $r^f_{\Delta t}$ of $X_{\Delta t}$, using also the previous observation, as follows
\begin{align*}
S_{\Delta t, j+1} &= S_{\Delta t,j} + S_{\Delta t,j} (r_{\Delta t,j}^d - r_{\Delta t,j}^f) \Delta \gamma_j + S_{\Delta t,j} \sqrt{V_{\Delta t,j}\, \Delta \gamma_j}\; H_{1,1}Z_1,\\
V_{\Delta t,j+1} &=V_{\Delta t,j} +\kappa_v (\theta_v - V_{\Delta t,j})\; \Delta\gamma_j + \sigma_v\sqrt{V_{\Delta t,j}\, \Delta\gamma_j} \; \phi_1,\\
r_{\Delta t,j+1}^d &= r_{\Delta t,j}^d+ \kappa_d (\theta_d - r_{\Delta t,j}^d) \ \Delta \gamma_j + \sigma_d \sqrt{r_{\Delta t,j}^d\, \Delta\gamma_j} \; \phi_2,\\
r_{\Delta t,j+1}^f &=  r_{\Delta t,j}^f + \left(\kappa_f  (\theta_f- r_{\Delta t,j}^f) - \sigma_f\rho_{sf} \sqrt{V_{\Delta t,j} r_{\Delta t,j}^f}\right)\Delta \gamma_j+ \sigma_f  \sqrt{r_{\Delta t,j}^f\,\Delta\gamma_j} \; \phi_3, 
\end{align*}
where
\begin{align*}
\phi_1 &=H_{2,1}Z_1+H_{2,2}  Z_2,\\
\phi_2 &=H_{3,1} Z_1+ H_{3,2}\; Z_2 + H_{3,3}\; Z_3,\\
\phi_3 &=H_{4,1} Z_1+ H_{4,2}\; Z_2 + H_{4,3}\; Z_3 + H_{4,4}\; Z_4
\end{align*} 
and $H_{i,j}$ are the components of the Cholesky factor of $\Sigma_{\bW}$, as in equation \eqref{eq:Choleskyfactor}.\\
At this point, we want to price American options numerically. Monte Carlo is a widespread method for option pricing as it can be used with any type of probability distribution. However, by construction, with Monte Carlo, it is difficult to derive the holding value  (or the continuation value)  at any time because that value depends on the unknown subsequent path. The research has dealt with the problem in several ways (see, for example, 
\cite{Margrabe, Carriere, Huang}, and references therein). Here we apply the Least Squares Monte Carlo simulation method (LSM) by Longstaff and Schwartz((\cite{Longstaff},\cite{Kavacs}, \cite{Samimi}), 
which consists of an algorithm for computing the price of an American option by stepping backward in time. Then, we compare the payoff from instant exercise with the expected payoff from continuation at any exercise time.  
Results are displayed in Figure ~\ref{F-1} where we show ten different paths of the asset price with parameters: $t_0=\delta t>0$, $T=5$, $\beta=0.5$,  $S_{t_0} = S_0=100$, $N=50$ time steps and $\Delta t=T-t_0/N$. 
In Figure ~\ref{F-2}, we display also ten simulated paths of the asset price under Heston-CIR of VG  L\'{e}vy type process for $t_0=\delta t>0$, $T=5$, $\beta=2$, $S_0=100$, $N=50$ time steps.

%%%%%%%%%%%%%%%%%%%%%%%%%%%%%%%%%%%%%%%%%%%%%%%%%%%%%%%%%%%%%%%%%%%%%%%%%%%
\begin{figure}[!htbp]
\centerline{\includegraphics[height=.27\paperheight]{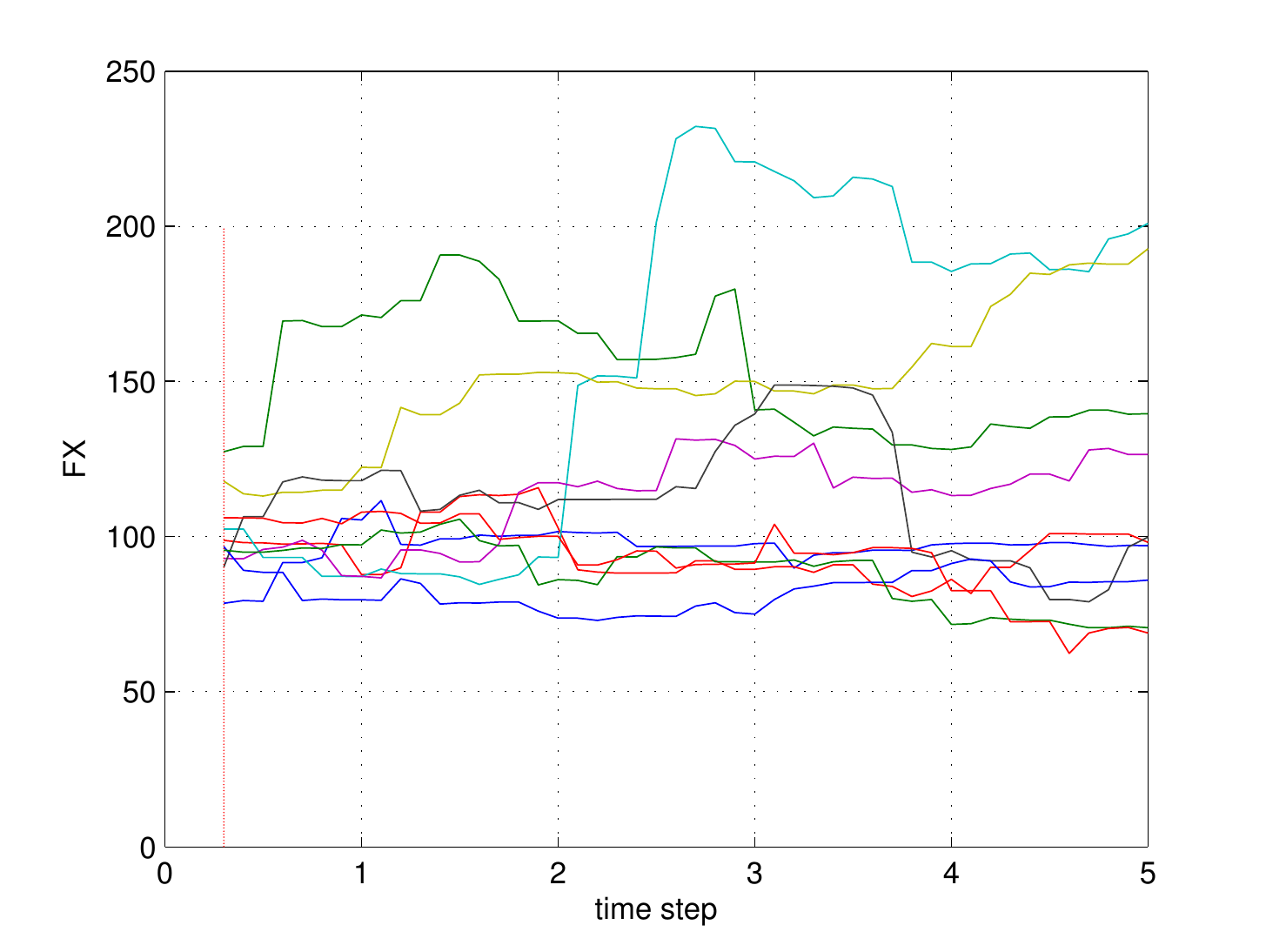}}
\caption{\small Simulated paths of the FX under Heston-CIR of VG  L\'{e}vy type process, with $T=5$, $\beta=0.5$,  and $S_0=100$.} \label{F-1} 
\end{figure}
%%%%%%%%%%%%%%%%%%%%%%%%%%%%%%%%%%%%%%%%%%%%%%%%%%%%%%%%%%%%%%%%%%%%%%%%%%%
\begin{figure}[!ht!]
\centerline{\includegraphics[height=.27\paperheight]{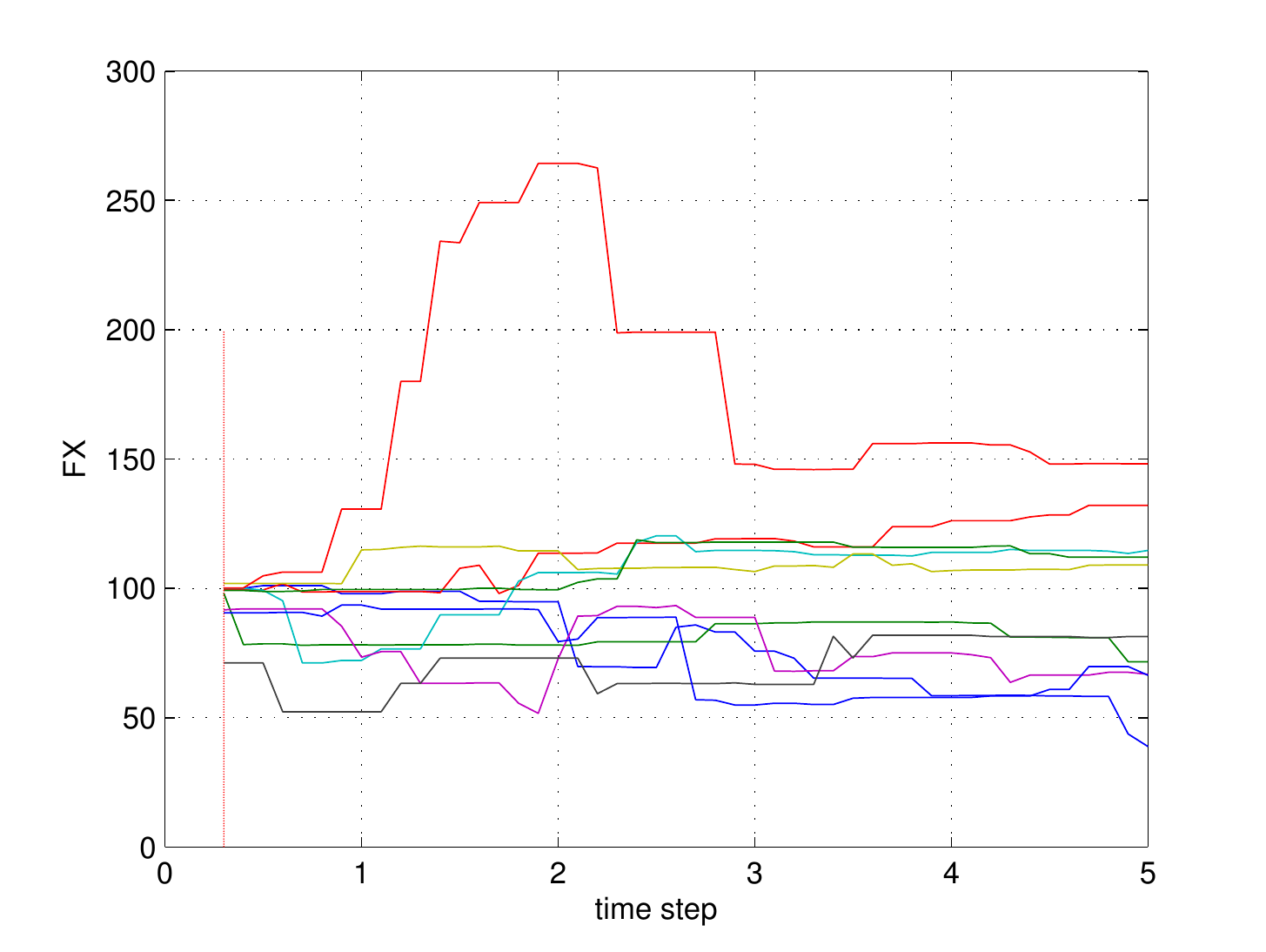}}
\caption{\small Simulated paths of the  FX under Heston-CIR of VG  L\'{e}vy tye process, with $T=5$, $\beta=2$, and $S_0=100$.} \label{F-2}
\end{figure}
%%%%%%%%%%%%%%%%%%%%%%%%%%%%%%%%%%%%%%%%%%%%%%%%%%%%%%%%%%%%%%%%%%%%%%%%%%%
%
In Figures ~\ref{F-3} and  ~\ref{F-4}, we have plotted the histogram of $S_t$ under Heston-CIR of VG  L\'{e}vy type process, with $\beta=0.5$ and $\beta=2$, respectively. The scale parameter $\beta$ controls the kurtosis of the distribution. Therefore, raising the parameter $\beta$ shifts mass to the tails.
 % We display this phenomenon in Figure ~\ref{F-5}. 
%
%%%%%%%%%%%%%%%%%%%%%%%%%%%%%%%%%%%%%%%%%%%%%%%%%%%%%%%%%%%%%%%%%%%%%%%%%%%
\begin{figure}[h]
\centerline{\includegraphics[height=.25\paperheight]{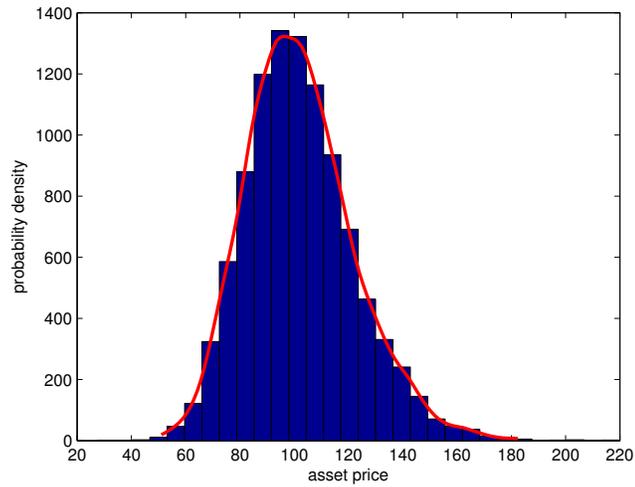}}
\caption{\small Histogram of the FX under the Heston-CIR of VG L\'{e}vy type model with $\beta=0.5$.} \label{F-3}
\end{figure}
%%%%%%%%%%%%%%%%%%%%%%%%%%%%%%%%%%%%%%%%%%%%%%%%%%%%%%%%%%%%%%%%%%%%%%%%%%%
\begin{figure}[h]
\centerline{\includegraphics[height=.25\paperheight]{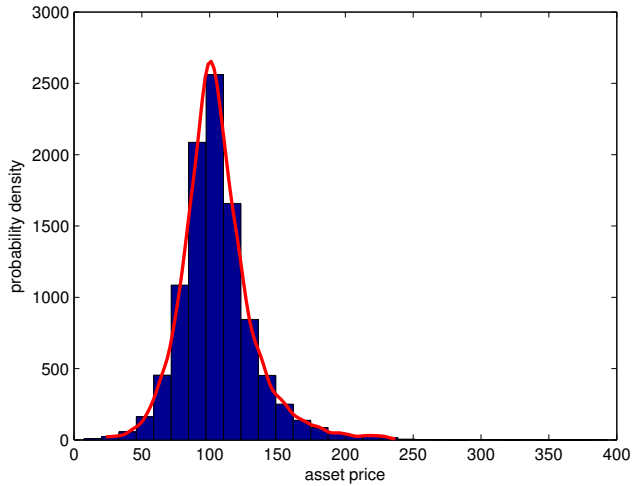}}
\caption{\small Histogram of the FX under the Heston-CIR of VG L\'{e}vy type model with $\beta=2$.} \label{F-4}
\end{figure}
%%%%%%%%%%%%%%%%%%%%%%%%%%%%%%%%%%%%%%%%%%%%%%%%%%%%%%%%%%%%%%%%%%%%%%%%%%%

%%%%%%%%%%%%%%%%%%%%%%%%%%%%%%%%%%%%%%%%%%%%%%%%%%%%%%%%%%%%%%%%%%%%%%%%%%%
%
%\par
%\newpage
\pagebreak
\bigskip\bigskip

\section{Dataset} \label{Sec:Dataset}

Concerning data, we retrieved interest rates and FX quotes from FRED, Federal Reserve Bank of St. Louis. In particular, the domestic rate is EURONTD156N  \cite{EURONTD156N} (Fig. \ref{Fig:EURONTD156N}), the foreign rate  is USDONTD156N \cite{USDONTD156N} (Fig. \ref{Fig:USDONTD156N}) and the FX is \cite{DEXUSEU}  (Fig. \ref{Fig:DEXUSEU}).

%%%%%%%%%%%%%%%%%%%%%%%%%%%%%%%%%%%%%%%%%%%%%%%%%%%%%%%%%%%%%%%%%%%%%%%%%%%

\begin{figure}[!htbp]
\centerline{\includegraphics[width=1\textwidth]{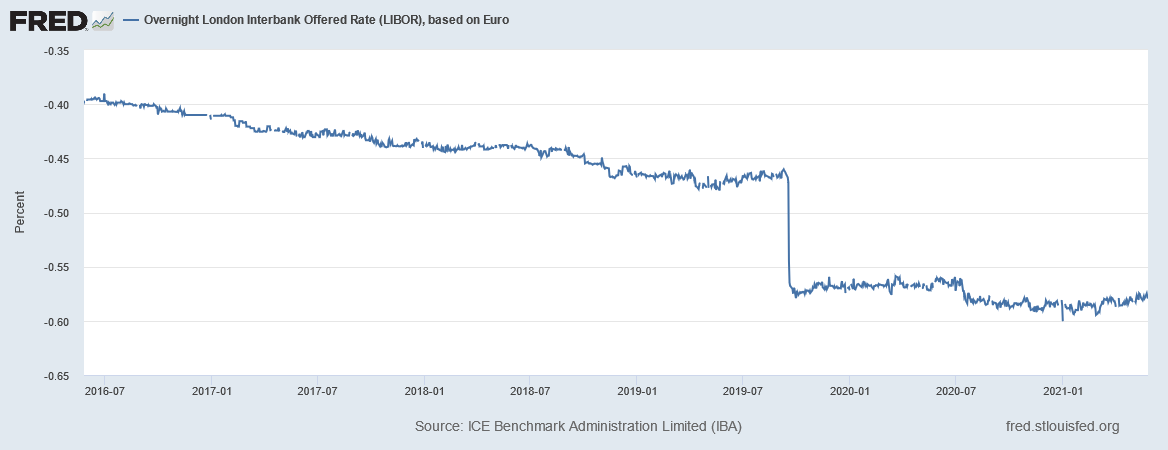}}
\caption{\small Overnight London Interbank Offered Rate (LIBOR), based on Euro [EURONTD156N]. Data from 2016-05-26 to 2021-05-26} \label{Fig:EURONTD156N} 
\end{figure}

\begin{figure}[!htbp]
\centerline{\includegraphics[width=1\textwidth]{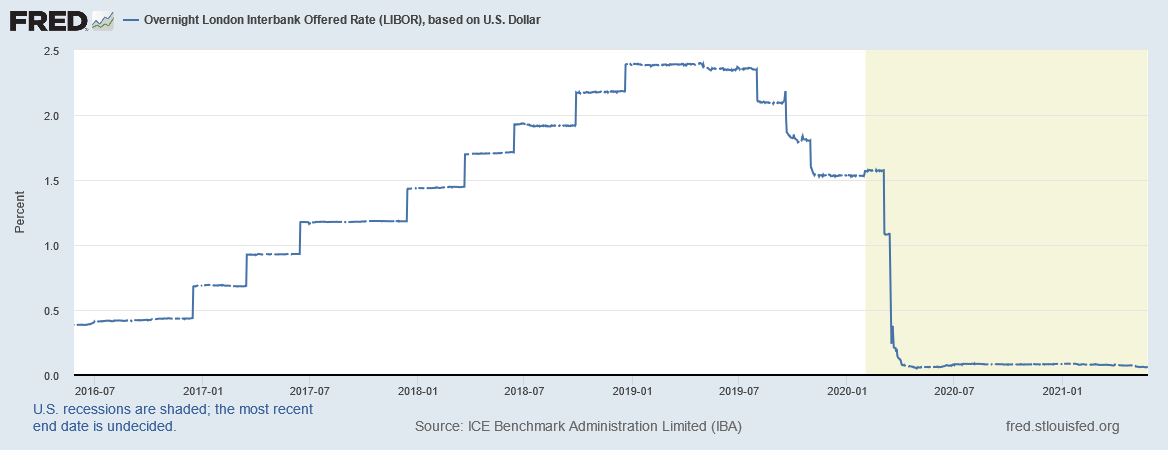}}
\caption{\small ICE Benchmark Administration Limited (IBA), Overnight London Interbank Offered Rate (LIBOR), based on U.S. Dollar [USDONTD156N]. Data from 2016-05-26 to 2021-05-26. The yellow strip to the right highlights the COVID-19 pandemic} \label{Fig:USDONTD156N} 
\end{figure}

%%%%%%%%%%%%%%%%%%%%%%%%%%%%%%%%%%%%%%%%%%%%%%%%%%%%%%%%%%%%%%%%%%%%%%%%%%%
\begin{figure}[!htbp]
\centerline{\includegraphics[width=1\textwidth]{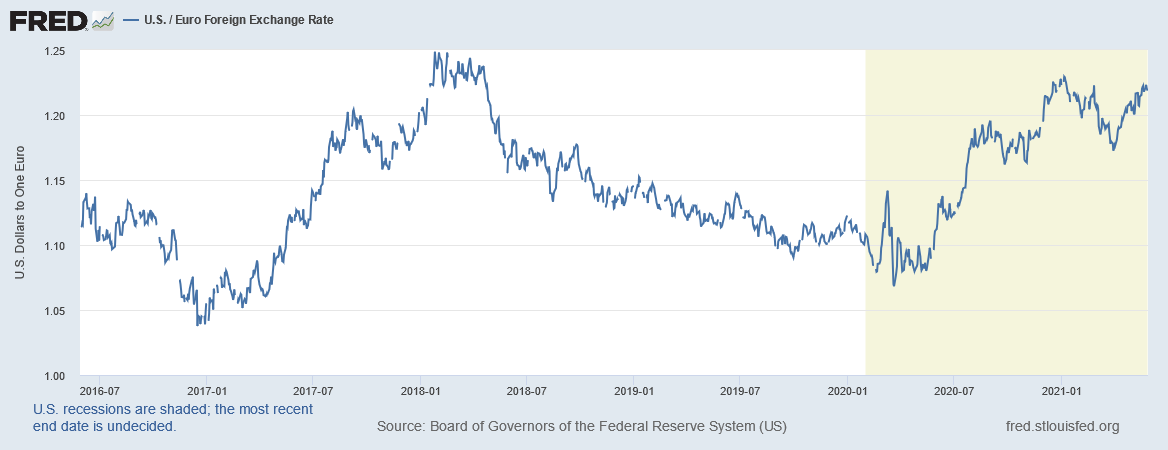}}
\caption{\small Board of Governors of the Federal Reserve System (US), U.S. / Euro Foreign Exchange Rate [DEXUSEU]. Data from 2016-05-26 to 2021-05-26. The yellow strip to the right highlights the COVID-19 pandemic}  \label{Fig:DEXUSEU} 
\end{figure}

American options' data were retrieved from the Chicago Mercantile Exchange (CME). Namely, on Wed, Jun 2nd, 2021, we have taken a snapshot of the following FX contracts: a)  Euro FX Sep '21 (E6U21), 36 days to expiration on 07/09/21, b) Euro FX Sep '21 (E6U21) 92 days to expiration on 09/03/21, c) Euro FX Dec '21 (E6Z21), 184 days to expiration on 12/03/21.

\section{Results} \label{Sec:NumericalExperiments}
Having discussed some stylized facts regarding returns and options' dynamics in Section \ref{Sec:StylFacts} concerning the VG models, in the following we report the results of our simulations on American options. 

Note that, while the proper risk-neutral simulations can be found in Madan et al. \cite{Madan1, Madan}, here we test our approach on real data.

%
%\subsection{Results for American options} \label{Sec:NumericalExperimentsforAmericanOption}
%
To demonstrate the efficiency of the LSM method, we present some tests on American put options (see also \cite{Samimi, Broadie, Dorbov} and references therein) under Heston-CIR L\'{e}vy model. We assign the following values, $S_0=100$, $V_0=0.0275$, $r^d_0=0.0524$, $r^f_0=0.0291$, $\kappa_v=1.70$, $\kappa_d=0.20$,  $\kappa_f=0.32$, $\theta_v=0.0232$, $\theta_d=0.0475$, $\theta_f=0.0248$, $\sigma_v=0.150$, $\sigma_d=0.0352$, $\sigma_f=0.0317$, $\rho_{sv}=-0.1$ , $\rho_{sd}=-0.15$, $\rho_{sf}=-0.15$, $\rho_{vd}=0.12$  and $T=1$. The number of simulations is 10,000, the number of steps is $N=50$ and we have repeated these calculations 100 times. 
 Tables~\ref{T-1}, \ref{T-2} and \ref{T-3}  present the results, where put option price, kurtosis and skewness are listed in columns and $\beta$ varies across rows. Note that a VG process has fat tails. As shown, the $\beta$ parameter is linked to the kurtosis and skewness. Therefore, $\beta$  allows us to control skewness and kurtosis, to obtain the peak of the distribution.  Moreover, according to the simulations, the higher the strike price, the higher the price of the American put options.

\begin{table}[!ht!]
\centering
\caption{\small Comparison of an American put option, kurtosis and skewness for different $\beta$ values, under the Heston-CIR of VG L\'{e}vy tye model for the FX market. The strike price value is $E=95$.} \label{T-1}
\begin{tabular}{ c  c  c  c   }
\hline
   &   option price  &  kurtosis & skewness \\ \hline
 $\beta=0.01$   & 2.5392  &  3.3980  &  0.5009  \\ 
\hline
  $\beta=0.1$    & 2.8064  &  4.2689 & 0.6127  \\
 \hline
 $\beta=0.5$  &  4.5454 & 6.3111 &  0.8331 \\
 \hline 
 $\beta=1$  & 4.3301 & 9.5439 & 1.0831 \\
  \hline
 $\beta=2$  &  4.7478 & 8.6184 & 0.9495  \\
  \hline
 $\beta=3$  & 4.0028& 12.7592 & 0.8370  \\
 \hline
\end{tabular}
\end{table}
%%%%%%%%%%%%%%%%%%%%%%%%%%%%%%%%%%
\begin{table}[!ht!]
\centering
\caption{\small Comparison of an American put option, kurtosis and skewness for different $\beta$ values, under Heston-CIR of VG L\'{e}vy tye model for the FX market. The strike price value is $E=100$.} \label{T-2}
\begin{tabular}{ c  c  c  c   }
\hline
   &   option price  &  kurtosis & skewness \\
\hline
 $\beta=0.01$& 2.3797 & 3.9047 & 0.5049 \\ 
\hline
  $\beta=0.1$& 3.0896 & 4.2724 & 0.6310  \\
 \hline
 $\beta=0.5$& 6.5128 & 7.1719 &  0.9343 \\
 \hline 
 $\beta=1$& 6.3791 & 10.3365 &  1.1680 \\
  \hline
 $\beta=2$& 6.0814 & 11.1584 & 1.0038 \\
  \hline
 $\beta=3$& 5.5649 & 17.0802 & 1.0214 \\
 \hline
\end{tabular}
\end{table}
%%%%%%%%%%%%%%%%%%%%%%%%%%%%%%%%%%
\begin{table}[!ht!]
\centering
\caption{\small Comparison of an American put option, kurtosis and skewness for different $\beta$ values under Heston-CIR of VG L\'{e}vy tye model for the FX market. The strike price value is $E=105$.} \label{T-3}
\begin{tabular}{ c  c  c  c   }
\hline
   &   option price  &  kurtosis & skewness \\
 \hline
 $\beta=0.01$ & 6.7369& 4.0933 & 0.6317 \\ 
\hline
  $\beta=0.1$ &5.3028  & 4.6504 & 0.6902\\
 \hline
 $\beta=0.5$ & 9.4756 & 6.0042 & 0.7517 \\
 \hline 
 $\beta=1$ & 9.2740 & 7.9159 & 0.8873 \\
  \hline
 $\beta=2$ & 8.6605 & 11.7784  & 0.9051\\
  \hline
 $\beta=3$ & 8.7161 & 34.1771 & 2.3745\\
 \hline
\end{tabular}
\end{table}
\newpage
 In Table~\ref{T-4}, we outlined the comparison of American put option prices in the FX market estimated by the Heston-CIR model and the Heston-CIR of VG L\'{e}vy type model for different strike values $E$. As displayed, the L\'{e}vy process has kurtosis.
%
%%%%%%%%%%%%%%%%%%%%%%%%%%%%%%%%%%
\begin{table}[!ht!]
\centering
\caption{\small Comparison of  American put option prices in the FX market by Heston-CIR model and Heston-CIR of VG L\'{e}vy model with different $E$.  The scale parameter value is $\beta=0.5$.} \label{T-4}
\begin{tabular}{ccccc}
\hline
    &  option price &  &  kurtosis &   \\
 \hline
 $  $   &  Heston-CIR  & L\'{e}vy &  Heston-CIR    & L\'{e}vy     \\ 
\hline
  $E=95$   & 2.3559 &4.5454 & 3.9467 & 6.2111  \\
 \hline
 $E=100$  & 3.7776 &   6.5128 &4.1502 & 10.3365    \\
 \hline 
 $E=105$  & 6.7438 & 9.4756  & 3.7244 & 6.0042 \\
  \hline
\end{tabular}
\end{table}

We employ $Euro\; FX\; Dec\; '21 \;(E6Z21)$ and $Euro\; FX\; Sep\; '21\; (E6U21)$ option prices to estimate the structure parameter of the Heston-CIR of VG L\'{e}vy type model by indirect inference. In Figures ~\ref{F-5}, ~\ref{F-6}, ~\ref{F-7}, and ~\ref{F-8} and Tables~\ref{T-6}, ~\ref{T-7}, ~\ref{T-8}, ~\ref{T-9}, ~\ref{T-10}, ~\ref{T-11}, ~\ref{T-12}, ~\ref{T-13} and ~\ref{T-14} we compare the price of American put and call options written on the $(E6Z21)$ and $(E6U21)$ with the corresponding estimated prices of the  Heston-CIR of VG L\'{e}vy type model, with maturity $T=182/365$ and $T=92/365$. The corresponding NRMSE (normalized root mean square error) has been also reported.  

 Madan and Seneta (1987) indicated a satisfactory fit of the VG process, using a  Chi-square goodness-of-fit test \cite{Madan2, Stephens}. The Chi-square goodness-of-fit test is a statistical tool used to evaluate whether data are suited to the considered model. We calculate the value of the Chi-square goodness of fit test using the following formula

\begin{align}
\boldsymbol{\chi} ^2=\sum _{i=1}^N \frac{(\mathbb{O}_i-\mathbb{E}_i)2}{\mathbb{E}_i}, \nonumber
\end{align}
where $N$ is the sample size, $\mathbb{O}_i$ are the observed counts and $\mathbb{E}_i$ are the expected counts. We apply the Chi-square goodness-of-fit test to the  FX market.
%%%%%%%%%%%%%%%%%%
\begin{table}[!ht!]
\centering
\caption{\small Chi-square Goodness-of-fit Test Statistics for Daily Log-returns. } \label{T-5}
\begin{tabular}{|c| c|c|c|c|}

 \hline
 $ T $   &  $t_0$ & $S_0$ & Chi-square L\'{e}vy  & Chi-square Normal  \\ 
\hline
 1/12   & 2021-6-4 & 1.0245 & 0.1704 & 0.2321 \\
 \hline
 1      & 9-6-2019 & 1.183 & 0.039 &  0.6803  \\
 \hline 
 2  & 9-6-2019 & 1.183 & 0.0141 & 1.132\\
  \hline
  5  & 4-9-2016 & 1.239 & 0.4693 & 3.218 \\
  \hline
\end{tabular}
\end{table}
In Table~\ref{T-5} goodness-of-fit test statistics are stated. The Heston-CIR of VG L\'{e}vy  model fits the data for FX markets with low Chi-square test statistics, meaning that the distance between theoretical and empirical distributions is small.
%

%%%%%%%%%%%%%%%%%%%%%%%%%%%%%%%%%%
\begin{table}[!ht!]
\centering
\caption{\small Comparison between estimated American put option prices with Heston-CIR of VG L\'{e}vy model and corresponding $FX$ market prices. $T=184/365$ and $\beta=0.5$.}  \label{T-6}
\begin{tabular}{ c  c  c  c  }
\hline
  &  simulated result  &  market price  & NRMSE \\ 
\hline
  $E=0.9$   & 4.2 $e^{-5}$
& 5 $e^{-5}$ & 2.719 $e^{-4}$\\
  \hline
  $E=0.92$   & 4.0801 $e^{-5}$  & 5 $e^{-5}$ & 2.908 $e^{-5}$ \\
  \hline
  $E=0.96$   & 4.782 $e^{-5}$  & 5 $e^{-5}$ & 2.9267 $e^{-4}$ \\
  \hline
  $E=0.98$   & 6.31 $e^{-5}$  & 5 $e^{-5}$ & 4.3435 $e^{-4}$ \\
  \hline
  $E=1.00$   & 1.0052 $e^{-4}$  & 1 $e^{-4}$ & 1.6836 $e^{-4}$ \\
  \hline
   $E=1.03$   & 1.6176 $e^{-4}$  & 1.5 $e^{-4}$ & 2.0751 $e^{-4}$ \\
  \hline
  $E=1.04$   & 1.83 $e^{-4}$  & 2 $e^{-4}$ & 2.8377 $e^{-4}$ \\
  \hline
   $E=1.05$   & 2.5126 $e^{-4}$  & 2.5 $e^{-4}$ & 4.1396 $e^{-4}$ \\
  \hline
\end{tabular}
\end{table}
%
%%%%%%%%%%%%%%%%%%%%%%%%%%%%%%%%%%
\begin{table}[!ht!]
\centering
\caption{\small Comparison between estimated American put option prices with Heston-CIR of VG L\'{e}vy model and corresponding $FX$ market prices. $T=184/365$ and $\beta=0.5$.}  \label{T-7}
\begin{tabular}{ c  c  c  c  }
\hline
 $  $   &  simulated result  &  market price  & NRMSE \\ 
\hline
  $E=1.13$   & 0.0015  & 0.0015 & 2.225 $e^{-4}$ \\
  \hline
  $E=1.135$   & 0.0017  & 0.0017 &1.0581 $e^{-4}$ \\
  \hline
  $E=1.14$   & 0.002  & 0.0019 & 2.7412 $e^{-4}$ \\
  \hline
  $E=1.145$   &  0.0024 & 0.0021 & 1.6858 $e^{-4}$ \\
  \hline
  $E=1.15$   &  0.0025 & 0.0024 & 1.4501 $e^{-4}$ \\
  \hline
   $E=1.155$   & 0.0027  & 0.0028 & 1.6588 $e^{-4}$ \\
  \hline
  $E=1.16$   & 0.0031  & 0.0032 & 1.3906 $e^{-4}$ \\
  \hline
   $E=1.165$   & 0.0032  & 0.0035 & 1.902 $e^{-4}$ \\
  \hline
\end{tabular}
\end{table}
%

%%%%%%%%%%%%%%%%%%%%%%%%%%%%%%%%%%
\begin{table}[!ht!]
\centering
\caption{\small Comparison between estimated American put option prices with Heston-CIR of VG L\'{e}vy model and corresponding $FX$ market prices. $T=184/365$ and $\beta=0.5$.}  \label{T-8}
\begin{tabular}{ c  c  c  c  }
\hline
   &  simulated result  &  market price  & NRMSE \\ 
\hline
  $E=1.39$   & 0.1634  & 0.1631 & 1.0619 $e^{-5}$ \\
  \hline
  $E=1.38$   & 0.1533  & 0.1532 &1.2562 $e^{-5}$ \\
  \hline
  $E=1.37$   & 0.1432  & 0.1433 & 3.0054 $e^{-6}$ \\
  \hline
  $E=1.36$   &  0.1333 & 0.1335 & 2.414 $e^{-5}$ \\
  \hline
  $E=1.35$   &  0.1234 & 0.1237 & 1.2835 $e^{-5}$ \\
  \hline
   $E=1.34$   & 0.1136  & 0.114 & 1.4544 $e^{-5}$ \\
  \hline
  $E=1.33$   & 0.1036  & 0.1043 & 1.4414 $e^{-5}$ \\
  \hline
   $E=1.32$   & 0.0937  & 0.0947 & 4.632 $e^{-5}$ \\
  \hline
\end{tabular}
\end{table}
%
%%%%%%%%%%%%%%%%%%%%%%%%%%%%%%%%%%
\begin{table}[!ht!]
\centering
\caption{\small Comparison between estimated American put option prices with Heston-CIR of VG L\'{e}vy model and corresponding $FX$ market prices. $T=92/365$ and $\beta=0.5$.}  \label{T-9}
\begin{tabular}{ c  c  c  c  }
\hline
 $  $   &  simulated result  &  market price  & NRMSE \\ 
\hline
  $E=1.31$   & 0.0853  & 0.0853 & 7.7691 $e^{-6}$ \\
  \hline
  $E=1.3$   & 0.0754  & 0.0755 & 1.6754 $e^{-6}$ \\
  \hline
  $E=1.29$   & 0.0651  & 0.0659 & 3.7741 $e^{-5}$ \\
  \hline
  $E=1.285$   &  0.0606 & 0.0611 & 1.4728 $e^{-5}$ \\
  \hline
  $E=1.28$   &  0.0557 & 0.0564 & 1.1434 $e^{-5}$ \\
  \hline
   $E=1.275$   & 0.0507  & 0.0518 & 3.0421 $e^{-5}$ \\
  \hline
  $E=1.185$   & 0.0023  & 0.0023 & 1.9313 $e^{-4}$ \\
  \hline
   $E=1.18$   & 0.0027  & 0.0028 & 6.1121 $e^{-5}$ \\
  \hline
\end{tabular}
\end{table}
%

%%%%%%%%%%%%%%%%%%%%%%%%%%%%%%%%%%
\begin{table}[!ht!]
\centering
\caption{\small Comparison between estimated American put option prices with Heston-CIR of VG L\'{e}vy model and corresponding $FX$ market prices. $T=92/365$ and $\beta=0.5$.}  \label{T-10}
\begin{tabular}{ c  c  c  c  }
\hline
 $  $   &  simulated result  &  market price  & NRMSE \\ 
\hline
  $E=1.39$   & 0.1647  & 0.1648 & 1.7288 $e^{-6}$ \\
  \hline
  $E=1.38$   & 0.1547  & 0.1548 &4.8363 $e^{-6}$ \\
  \hline
  $E=1.37$   & 0.1446  & 0.1448 & 1.1436 $e^{-5}$ \\
  \hline
  $E=1.36$   &  0.1348 & 0.1347 & 4.6565 $e^{-6}$ \\
  \hline
  $E=1.35$   &  0.1249 & 0.1249 & 6.4369 $e^{-6}$\\
  \hline
   $E=1.34$   & 0.1148  & 0.1149 & 3.411 $e^{-6}$ \\
  \hline
  $E=1.33$   & 0.1049  & 0.105 & 1.3209 $e^{-5}$ \\
  \hline
   $E=1.32$   & 0.0949  & 0.0951 & 1.1516 $e^{-5}$ \\
  \hline
\end{tabular}
\end{table}
%

%%%%%%%%%%%%%%%%%%%%%%%%%%%%%%%%%%
\begin{table}[!ht!]
\centering
\caption{\small Comparison between estimated American call option prices with Heston-CIR of VG L\'{e}vy model and corresponding $FX$ market prices. $T=184/365$ and $\beta=0.5$.}  \label{T-11}
\begin{tabular}{ c  c  c  c  }
\hline
 $  $   &  simulated result  &  market price  & NRMSE \\ 
\hline
  $E=0.9$   & 0.3278  & 0.3274 & 3.4218 $e^{-5}$ \\
  \hline
  $E=0.91$   & 0.3170 & 0.3174 & 3.042 $e^{-5}$ \\
  \hline
  $E=0.92$   &  0.3074 & 0.3074 & 3.4581 $e^{-5}$ \\
  \hline
  $E=0.93$   & 0.2972  & 0.2974 & 4.1098 $e^{-5}$ \\
  \hline
  $E=0.94$   &  0.2871 & 0.2874 & 4.262 $e^{-5}$ \\
  \hline
   $E=0.95$   & 0.2778  & 0.2774 & 5.5471 $e^{-5}$ \\
  \hline
  $E=0.96$   &  0.2691 &0.2675  & 5.4698 $e^{-5}$ \\
  \hline
   $E=0.97$   & 0.2584  & 0.2575 & 5.3664 $e^{-5}$ \\
  \hline
\end{tabular}
\end{table}
%
%%%%%%%%%%%%%%%%%%%%%%%%%%%%%%%%%%

%%%%%%%%%%%%%%%%%%%%%%%%%%%%%%%%%%
\begin{table}[!ht!]
\centering
\caption{\small Comparison between estimated American call option prices with Heston-CIR of VG L\'{e}vy model and corresponding $FX$ market prices. $T=184/365$ and $\beta=0.5$.}  \label{T-12}
\begin{tabular}{ c  c  c  c  }
\hline
 $  $   &  simulated result  &  market price  & NRMSE \\ 
\hline
  $E=1.39$   & 5.5704 $e^{-4}$ & 5 $e^{-4}$ & 4.0065 $e^{-5}$ \\
  \hline
  $E=1.38$   & 7.8713 $e^{-4}$ & 7 $e^{-4}$ & 3.2333 $e^{-5}$ \\
  \hline
  $E=1.37$   & 8.5784 $e^{-4}$  & 8 $e^{-4}$  & 4.005  $e^{-5}$ \\
  \hline
  $E=1.36$   & 9.9123 $e^{-4}$  & 0.001 & 1.5721 $e^{-5}$ \\
  \hline
  $E=1.35$   & 0.0012 & 0.0012 & 5.2441 $e^{-5}$ \\
  \hline
   $E=1.34$   & 0.0014  & 0.0014 & 1.2596 $e^{-5}$\\
  \hline
  $E=1.33$   &  0.0016 & 0.0018 & 2.2165 $e^{-5}$\\
  \hline
   $E=1.32$   & 0.0021  & 0.0022 & 2.8563 $e^{-5}$ \\
  \hline
\end{tabular}
\end{table}
%
%%%%%%%%%%%%%%%%%%%%%%%%%%%%%%%%%%
\begin{table}[!ht!]
\centering
\caption{\small Comparison between estimated American call option prices with Heston-CIR of VG L\'{e}vy model and corresponding $FX$ market prices. $T=92/365$ and $\beta=0.5$.}  \label{T-13}
\begin{tabular}{ c  c  c  c  }
\hline
 $  $   &  simulated result  &  market price  & NRMSE \\ 
\hline
  $E=0.9$   & 0.3259  & 0.3253 & 1.6764 $e^{-5}$\\
  \hline
  $E=0.91$   & 0.3159 & 0.3153 & 7.9192 $e^{-6}$ \\
  \hline
  $E=0.92$   &  0.3054 & 0.3053 & 4.9681 $e^{-6}$ \\
  \hline
  $E=0.93$   & 0.2956  & 0.2953 & 8.3867 $e^{-6}$ \\
  \hline
  $E=0.94$   &  0.2854 & 0.2853 & 1.0489 $e^{-5}$ \\
  \hline
   $E=0.95$   & 0.2752  & 0.2753 & 5.935 $e^{-6}$ \\
  \hline
  $E=0.96$   &  0.2653 &0.2653  & 8.2524 $e^{-6}$ \\
  \hline
   $E=0.97$   & 0.2553  & 0.2553 & 1.2934 $e^{-5}$\\
  \hline
\end{tabular}
\end{table}
%
%%%%%%%%%%%%%%%%%%%%%%%%%%%%%%%%%%
%%%%%%%%%%%%%%%%%%%%%%%%%%%%%%%%%%
\begin{table}[!ht!]
\centering
\caption{\small Comparison between estimated American call option prices with Heston-CIR of VG L\'{e}vy model and corresponding $FX$ market prices. $T=92/365$ and $\beta=0.5$.}  \label{T-14}
\begin{tabular}{ c  c  c  c  }
\hline
 $  $   &  simulated result  &  market price  & NRMSE \\ 
\hline
  $E=1.39$   & 1.1868 $e^{-4}$ & 1 $e^{-4}$ & 2.4010 $e^{-4}$ \\
  \hline
  $E=1.38$   & 1.7458 $e^{-4}$  & 1 $e^{-4}$ & 6.3155 $e^{-5}$ \\
  \hline
  $E=1.37$   & 1.8853 $e^{-4}$  & 1 $e^{-4}$  & 9.7082 $e^{-5}$ \\
  \hline
  $E=1.36$   & 1.5572 $e^{-4}$  & 1.5 $e^{-4}$ & 9.7195 $e^{-5}$ \\
  \hline
  $E=1.35$   & 1.8542 $e^{-4}$ & 2 $e^{-4}$ & 2.7133 $e^{-5}$ \\
  \hline
   $E=1.34$   & 2.2579 $e^{-4}$  & 2.5 $e^{-4}$& 1.0322 $e^{-5}$ \\
  \hline
  $E=1.33$   &  2.8955 $e^{-4}$ & 3 $e^{-4}$ & 1.868 $e^{-5}$\\
  \hline
   $E=1.32$   & 5.6845 $e^{-4}$  & 4.5 $e^{-4}$ & 2.9546 $e^{-5}$\\
  \hline
\end{tabular}
\end{table}
%

%%%%%%%%%%%%%%%%%%%%%%%%%%%%%%%%%%%%%%%%%%%%%%%%%%%%%%%%%%%%%%%%%%%%%%%%%%%
\begin{figure}[!htbp]
\centerline{\includegraphics[height=.27\paperheight]{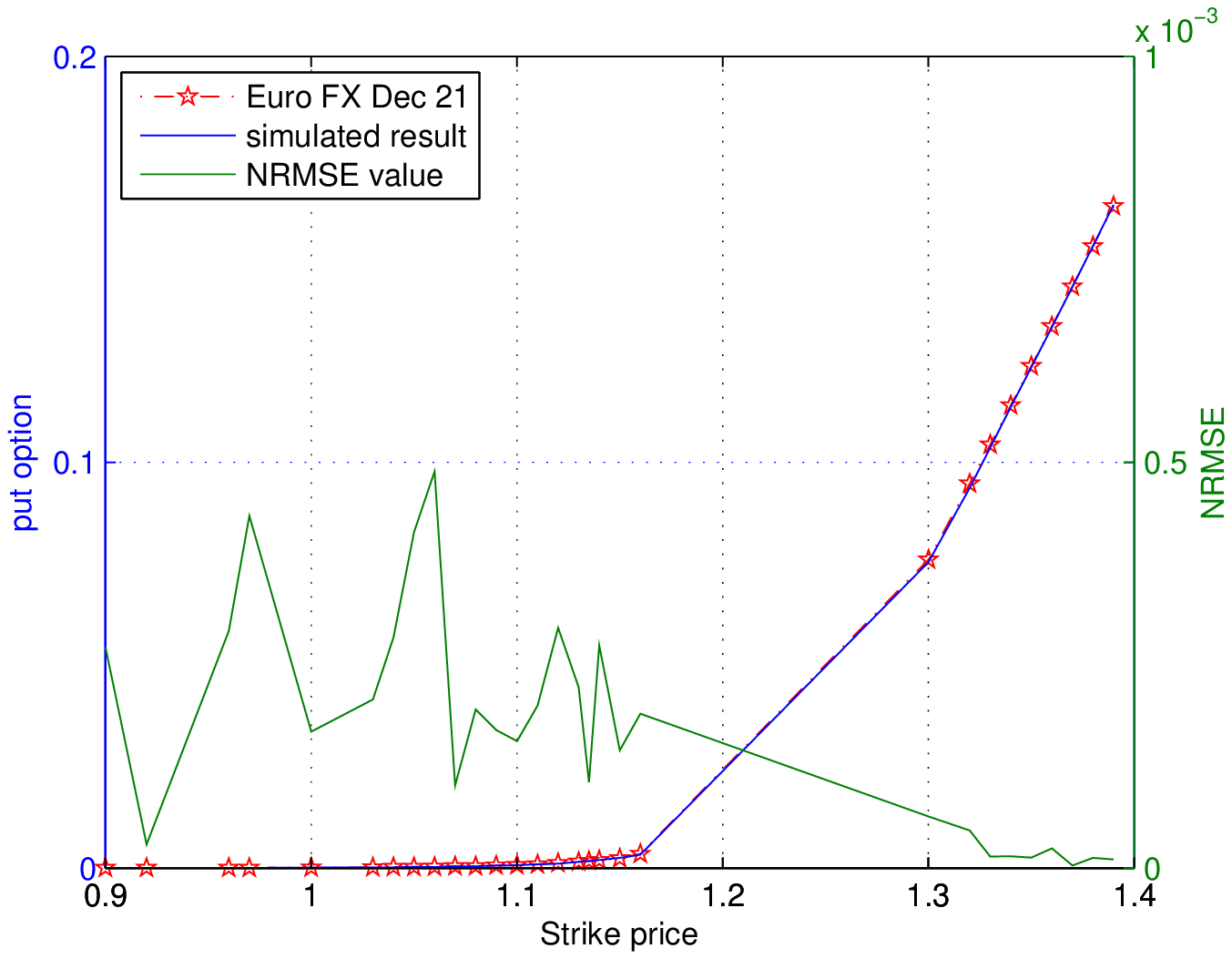}}
\caption{\small Simulated American put option prices  (blue line) and market values (red line) with maturity $T=184$ days.} \label{F-5}
\end{figure}
%%%%%%%%%%%%%%%%%%%%%%%%%%%%%%%%%%%%%%%%%%%%%%%%%%%%%%%%%%%%%%%%%%%%%%%%%%%
%%%%%%%%%%%%%%%%%%%%%%%%%%%%%%%%%%%%%%%%%%%%%%%%%%%%%%%%%%%%%%%%%%%%%%%%%%%
\begin{figure}[!htbp]
\centerline{\includegraphics[height=.27\paperheight]{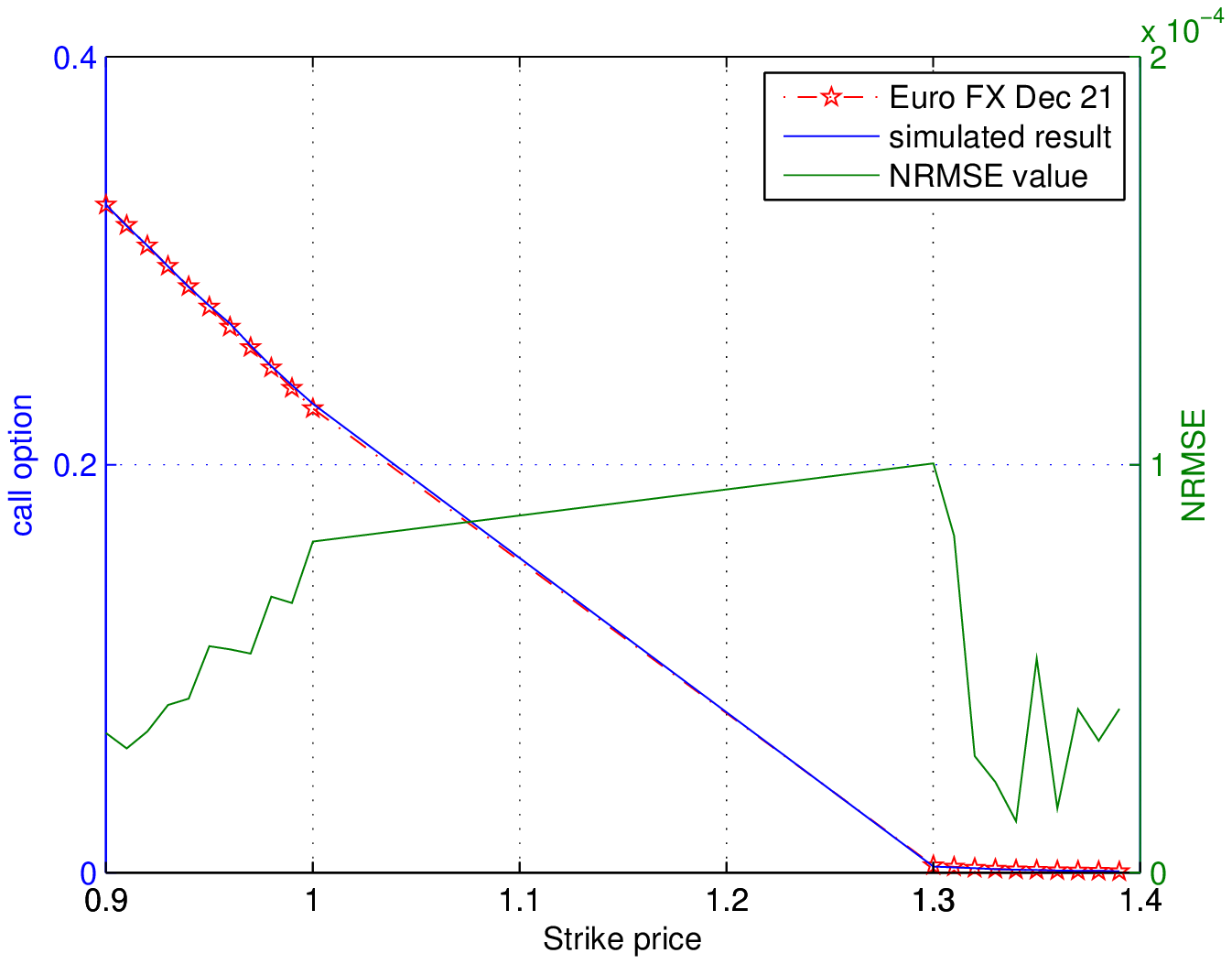}}
\caption{\small Simulated  American call option prices (blue line) and market values (red line) with maturity $T=184$ days.} \label{F-6}
\end{figure}
%%%%%%%%%%%%%%%%%%%%%%%%%%%%%%%%%%%%%%%%%%%%%%%%%%%%%%%%%%%%%%%%%%%%%%%%%%%
%%%%%%%%%%%%%%%%%%%%%%%%%%%%%%%%%%%%%%%%%%%%%%%%%%%%%%%%%%%%%%%%%%%%%%%%%%%
\begin{figure}[!htbp]
\centerline{\includegraphics[height=.27\paperheight]{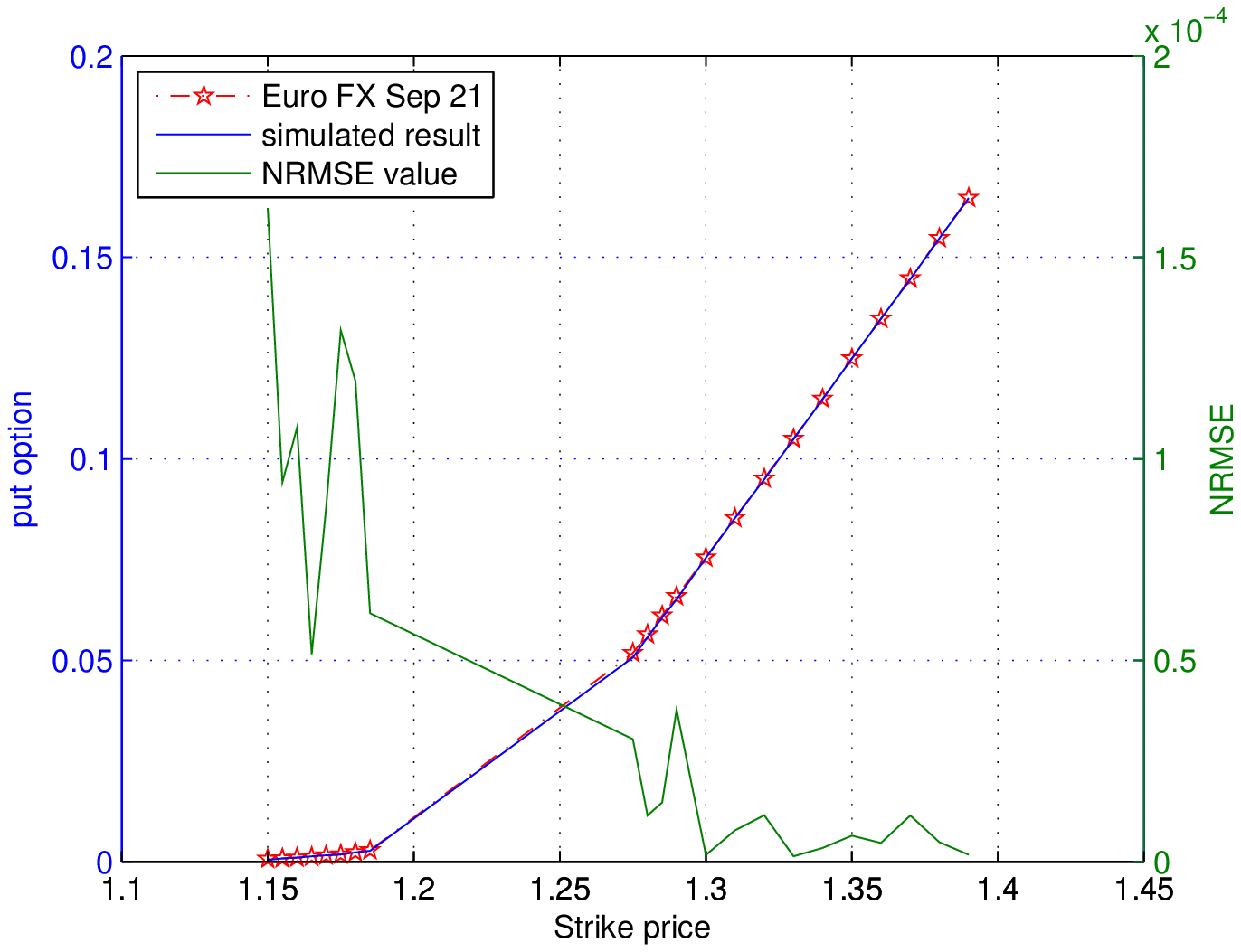}}
\caption{\small Simulated American put option prices (blue line) and market values (red line) with maturity $T=92$ days.} \label{F-7}
\end{figure}
%%%%%%%%%%%%%%%%%%%%%%%%%%%%%%%
%%%%%%%%%%%%%%%%%%%%%%%%%%%%%%%%%%%%%%%%%%%%%%%%%%%%%%%%%%%%%%%%%%%%%%%%%%%
\begin{figure}[!htbp]
\centerline{\includegraphics[height=.27\paperheight]{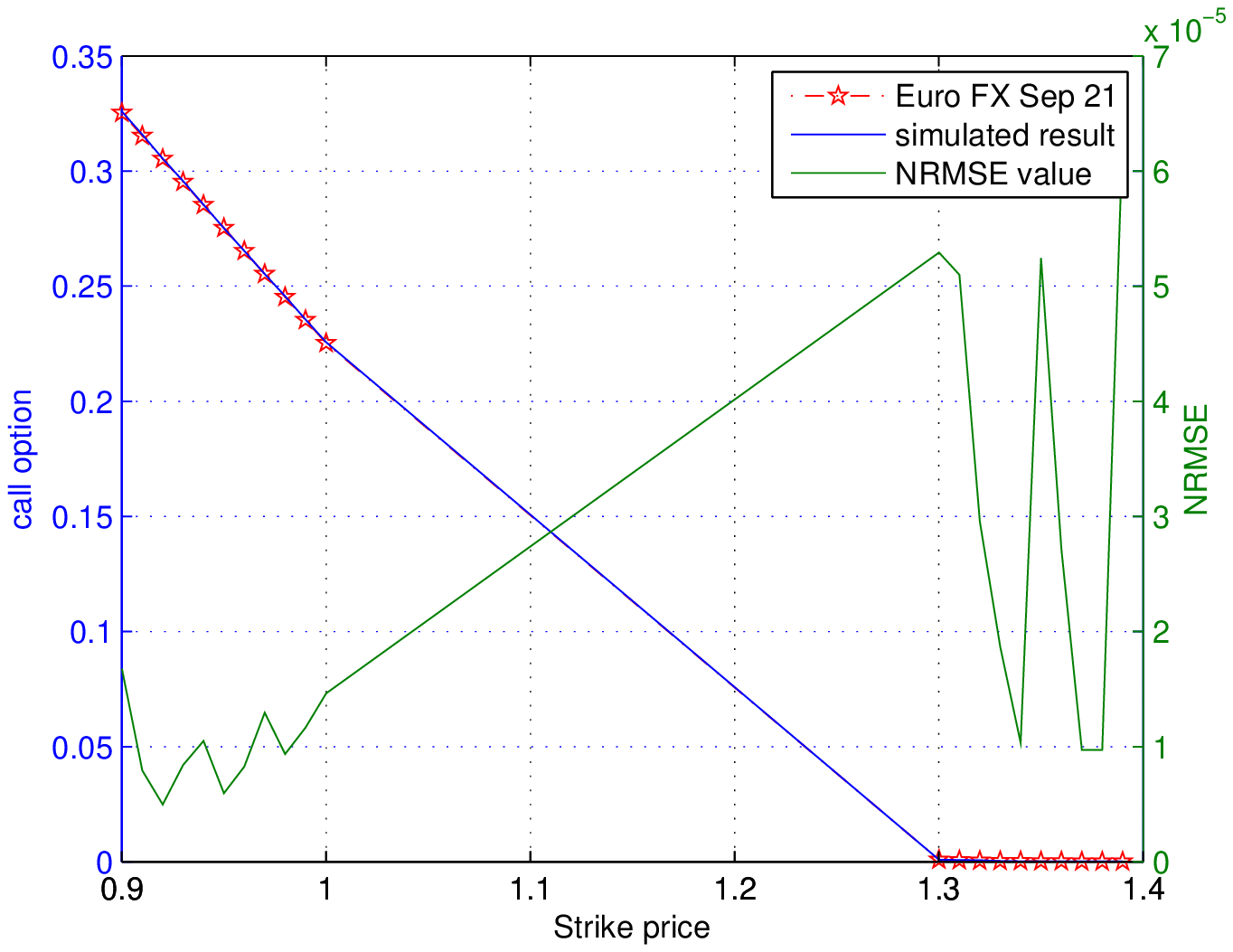}}
\caption{\small Simulated  American call option prices (blue line) and market values (red line) with maturity $T=92$ days.} \label{F-8}
\end{figure}
%%%%%%%%%%%%%%%%%%%%%%%%%%%%%%%%%%%%%%%%%%%%%%%%%%%%%%%%%%%%%%%%%%%%%%%%%%%
%%%%%%%%%%%%%%%%%%%%%%%%%%%%%%%%%%%%%%%%%%%%%%%%%%%%%%%%%%%%%%%%%%%%%%%%%%%
\begin{figure}[!htbp]
\centerline{\includegraphics[height=.27\paperheight]{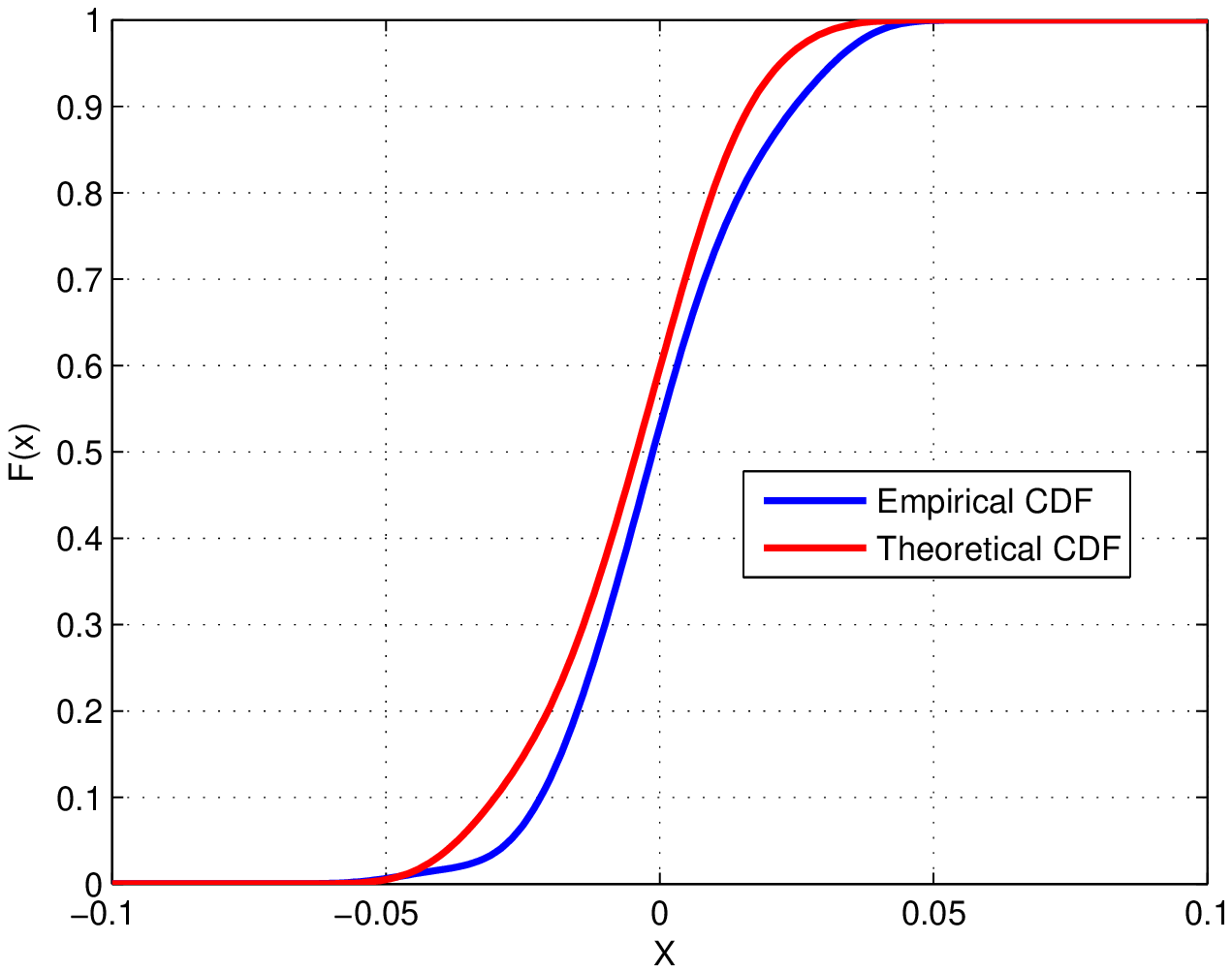}}
\caption{\small  The red line is an empirical CDF, and the blue line is a Market CDF. $T=5$ years.} \label{F-9}
\end{figure}
%%%%%%%%%%%%%
%%%%%%%%%%%%%%%%%%%%%%%%%%%%%%%%%%%%%%%%%%%%%%%%%%%%%%%%%%%%%%%%%%%%%%%%%%%
\begin{figure}[!htbp]
\centerline{\includegraphics[height=.27\paperheight]{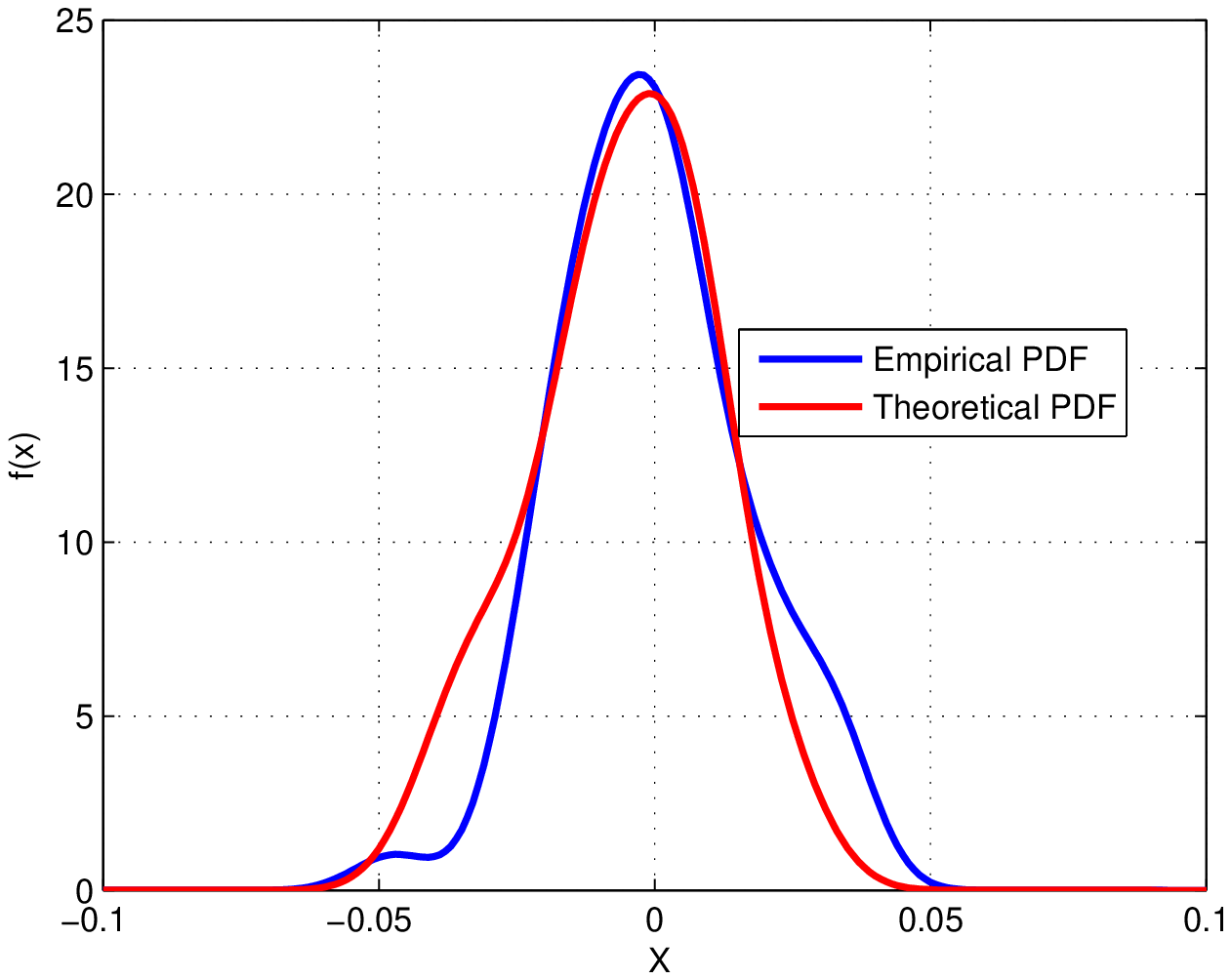}}
\caption{\small The red line is an empirical PDF, and the blue line is a Market PDF. $T=5$ years.} \label{F-10}
\end{figure}
%%%%%%%%%%%%%
%%%%%%%%%%%%%%%%%%%%%%%%%%%%%%%%%%%%%%%%%%%%
% 
\pagebreak
\newpage
\clearpage

\section{Conclusions} \label{Sec:Conclusions}
This paper, first, introduced Heston-CIR L\'{e}vy model for the FX market. The model has many positive features, for instance, through the parameters we can control tails, peaks and asymmetry. 
The proof of the strong convergence of the L\'{e}vy process with stochastic domestic short interest rates, foreign short interest rates and stochastic volatility was provided. 
Then the Euler-Maruyama discretization scheme was used to estimate the paths for this model. Finally, American put option for FX market under Heston-CIR L\'{e}vy process with the LSM method were priced and a test on real data was performed. The simulations prove that the chosen model fits well real-world data. \\\\

\appendix

%%%%%%%%%%%
\section{Proof of Lemma \ref{lem:Levy}}

\begin{proof} \label{proof:Levy}
Let us observe that the process $X(t):=(t,W_j(t))$ is a L\'evy process and $\gamma(t)$ is a subordinator independent of it. Then, by Phillips subordination theorem (see \cite[Theorem 30.1]{Sato}), $(\gamma(t),\Gamma_j(t))=X(\gamma(t))$ is still a L\'evy process. Still by \cite[Theorem 30.1]{Sato}, we know that for any Borel set $A \subset \R^2\setminus \{0\}$ it holds
\begin{equation}\label{eq:Levy1}
    \nu(A)=\int_{0}^{+\infty}\bQ((s,W_j(s))\in A)g_\gamma(s)ds.
\end{equation}
Now let us consider any couple of Borel set $A_1 \subset \R\setminus \{0\}$ and $A_2 \subset \R \setminus \{0\}$ and let $A=A_1 \times A_2$. It is clear that
\begin{equation*}
    \bQ((s,W_j(s))\in A)=\bQ((s,W_j(s))\in (A_1 \cup [0,+\infty))\times A_2),
\end{equation*}
hence, without loss of generality, we can assume $A_1 \subset [0,+\infty)$. Observe that
\begin{equation*}
\bQ((s,W_j(s))\in A)=\bQ(s \in A_1)\bQ(W_j(s) \in A_2)=1_{A_1}(s)\int_{A_2}p_1(s,u)du.
\end{equation*}
Thus, Equation \eqref{eq:Levy1} becomes
\begin{align*}
    \nu_1(A)&=\int_{0}^{+\infty}1_{A_1}(s)\left(\int_{A_2}p_1(s,u)du\right)g_\gamma(s)ds\\
    &=\int_{A_1}\left(\int_{A_2}p_1(s,u)du\right)g_\gamma(s)ds\\
    &=\int_{A}p_1(s,u)g_\gamma(s)dsdu,
\end{align*}
concluding the proof.
\end{proof}

%%%%%%%%%%%%%%
\section{Proof of Theorem \ref{thm:exunV0rd0}}

\begin{proof} \label{proof:exunV0rd0}

Let us just prove the existence and uniqueness statement for \eqref{eq:V0}, as the other one can be proven with the same rationale. First, we need to obtain the L\'evy-It\^o decomposition of $(\gamma(t),\Gamma_2(t))$. To do this, let us consider the Poisson random measure $N_{\Gamma_2}(t,dsdu)$ of the process $(\gamma(t),\Gamma_2(t))$, defined on $\R^2$ as
\begin{equation*}
    N_{\Gamma_2}(t,A)=\sum_{0 \le s \le t} 1_{A}((\Delta\gamma(s),\Delta \Gamma_2(s)) \in A), \ \forall A \in \cB(\R^2\setminus\{0\}),
\end{equation*}
where the summation makes sense since in any interval $[0,t]$ a L\'evy process admits countable (but possibly dense) jumps. Let also $\wN_{\Gamma_2}(dt,dsdu)=N_{\Gamma_2}(dt,dsdu)-\nu_1(dsdu)dt$ be the compensated Poisson random measure, where $\nu_1$ has been identified in Lemma \ref{lem:Levy}. Recalling that $\E[|(\gamma(t),\Gamma_2(t))|]<\infty$, we get the L\'evy-It\^o decomposition (see for instance \cite[Theorems 1.7-1.8]{Oksendal})
\begin{equation*}
    (\gamma(t),\Gamma_2(t))=\int_{\R^2}(s,u)\wN_{\Gamma_2}(t,dsdu).
\end{equation*}
This means that we can recast Equation \eqref{eq:V0} as follows.
\begin{equation}\label{eq:V1}
dV(t)=\kappa_v(a_v-V(t-))dt+\int_{\R^2}(\theta_vs+\sigma_vu\sqrt{V(t-)})\wN_{\Gamma_2}(dt,dsdu), \qquad V(0)=V_0>0.
\end{equation}
Now let us set, for $(x,s,u) \in \R^3$,
\begin{align*}
    b_V(x)&=\kappa_v(a_v-x)\\
    g_V(s,u,x)&=\theta_vs1_{\R^+}(s)+\sigma_vu1_{\R^+}(x)\sqrt{x},
\end{align*}
and let us consider the auxiliary Equation
\begin{equation}\label{eq:V1aux}
dV(t)=b_V(V(t-))dt+\int_{\R^2}g_V(s,u,V(t-))\wN_{\Gamma_2}(dt,dsdu), \qquad V_0>0.
\end{equation}
%Clearly, by hypotheses, $b_V(x)\ge 0$ for any $x \le 0$ and $g_V(x,u,x)=\theta_vs1_{\R^+}(s)$ for any $(s,u,x) \in \R^3$ such that $x \le 0$.
%Now let us observe that
%\begin{align}\label{eq:control2}
%\begin{split}
%    \int_{\R^2}g^2_V(s,u,x)\nu_1(dsdu) &= \theta^2_v\int_{\R^2}s^2\nu_1(dsdu)+\theta_v\sigma_v\sqrt{x}1_{\R^+}(x)\int_{\R^2}su\nu_1(dsdu)+\sigma^2_v x1_{\R^+}(x)\int_{\R^2}u^2\nu_1(dsdu)\\
%    &=\theta^2_vI_1+\theta_v\sigma_v\sqrt{x}1_{\R^+}(x)I_2+\sigma^2_vx1_{\R^+}(x)I_3.
%\end{split}
%\end{align}
%Let us evaluate $I_1$. It holds, by definition of $\nu_1$,
%\begin{equation*}
%    I_1=\alpha\int_{0}^{+\infty}se^{-\beta s}\left(\int_{\R}p_1(s,u)du\right)ds=\alpha\int_{0}^{+\infty}se^{-\beta s}ds=\frac{\alpha}{\beta^2}.
%\end{equation*}
%Recalling also $\int_{\R}up_1(s,u)du=0$, we conclude that $I_2=0$. Finally, we can evaluate $I_3$ by observing that $\int_{\R}u^2p_1(s,u)du=s$ and then proceeding as follows
%\begin{equation*}
%    I_3=\alpha \int_0^{+\infty}\frac{e^{-\beta s}}{s}\int_{\R}u^2p_1(s,u)duds=\alpha \int_0^{+\infty}e^{-\beta s}=\frac{\alpha}{\beta}<+\infty.
%\end{equation*}
%Hence Formula \eqref{eq:control2} can be recast as as
%\begin{equation*}
%    \int_{\R^2}g^2_V(s,u,x)\nu_1(dsdu)=\frac{\alpha \theta^2_v}{\beta^2}+\frac{\alpha \sigma_v^2}{\beta}x1_{\R^+}(x).
%\end{equation*}
%In conclusion we get that, clearly, 
%\begin{equation*}
%    2xb_V(x)+\int_{\R^2}g^2_V(s,u,x)\nu_1(dsdu)=\kappa_v(a_v-x)x+\frac{\alpha \theta^2_v}{\beta^2}+\frac{\alpha \sigma_v^2}{\beta}x1_{\R^+}(x)\le C(1+x^2),
%\end{equation*}
By direct evaluation, one can check that for any $x \in \R$
\begin{equation*}
    2xb_V(x)+\int_{\R^2}g^2_V(s,u,x)\nu_1(dsdu)\le C(1+x^2),
\end{equation*}
for some constant $C>0$. Indeed,
\begin{equation*}
    xb_V(x)=\kappa_va_vx-\kappa_vx^2\le C(1+x^2).
\end{equation*}
On the other hand
\begin{align*}
    g_V^2(s,u,x)&=(\theta_v s 1_{\R^+}(s)+\sigma_v u 1_{\R^+}(x)\sqrt{x})^2\\
    &\le 2\theta_v^2s^2+2\sigma_v^2u^2|x|
\end{align*}
and then
\begin{align}\label{eq:inequ}
\begin{split}
    \int_{\R^2}&g_V^2(s,u,x)\nu_1(dsdu)\le 2\theta_v^2\int_{\R^2}s^2\nu_1(dsdu)+2\sigma_v^2|x|\int_{\R^2}u^2\nu_1(dsdu)\\
    &=2\alpha\theta_v^2\int_{\R^2}se^{-\beta s}ds+2\alpha\sigma_v^2|x|\int_{\R^+}s^{-1}e^{-\beta s}\left(\int_{\R} u^2 p_1(s,u)du\right)ds\\
    &=2\alpha\theta_v^2\int_{\R^2}se^{-\beta s}ds+2\alpha\sigma_v^2|x|\int_{\R^+}e^{-\beta s}ds\le C(1+|x|) \le C(1+x^2).
\end{split}
\end{align}
This guarantees that any solution of Equation \eqref{eq:V1aux} is non-explosive by means of Theorem \cite[Theorem 2.2]{Xi}, since $b_V$ and $g_V$ satisfy \cite[Assumption 2.1]{Xi}. One can also check, still by direct evaluation, that for any $x,y \in \R$ it holds
\begin{equation}\label{eq:nonLipcond}
    \int_{\R^2}|g_V(s,u,x)-g_V(s,u,y)|^2\nu_1(dsdu)\le C|x-y|.
\end{equation}
for a constant $C>0$.
Now let us consider a sequence $\psi^{(1)}_n \in C^\infty_c(\R)$ such that ${\rm supp}(\psi^{(1)}_n) \subset (-n-1,n+1)$, $\psi^{(1)}_n(x)=1$ for any $x \in (-n,n)$ and $0 \le \psi_n(x) \le 1$ for any $x \in  \R$. Set, for any $(s,u,x) \in \R^3$, $b_n(x)=\psi^{(1)}_n(x) b_V(x)$ and $g_n(x)=\psi^{(1)}_n(x)g_V(s,u,x)$. It is clear that it still holds
\begin{equation}\label{eq:nonexp}
    2xb_n(x)+\int_{\R^2}g^2_V(s,u,x)\nu_1(dsdu)\le C(1+x^2),
\end{equation}
for some constant $C>0$. Let us also stress that, by Inequality \eqref{eq:inequ},
\begin{equation}\label{eq:sublingr}
    \int_{\R^2}g^2_V(s,u,x)\nu_1(dsdu)\le C(1+|x|),
\end{equation}
for a suitable constant $C>0$. In particular, Inequality \eqref{eq:nonexp} implies that if we consider the Equation
\begin{equation}\label{eq:V1auxn}
dV_n(t)=b_n(V_n(t-))dt+\int_{\R^2}g_n(s,u,V_n(t-))\wN_{\Gamma_2}(dt,dsdu), \qquad V_n(0)=V_0>0,
\end{equation}
its strong solutions are non-explosive. Now let us show that Equation \eqref{eq:V1aux} admits at least a weak solution. For any function $f \in C^2_b(\R)$, we can consider the operator
\begin{equation*}
    \cL_nf(x)=f'(x)b_n(x)+\int_{\R^2}(f(x+g_n(s,u,x))-f(x)-f'(x)g_n(s,u,x))\nu_1(dsdu).
\end{equation*}
Let us consider, for any $x \in \R$ and any Borel set $B \in \cB(\R)$,
\begin{equation*}
M(B,x)=\nu_1(\{(s,u) \in \R^2: \ g_n(s,u,x) \in B\}),
\end{equation*}
so that, by the change of variable formula (see \cite[Theorem 3.6.1]{Bogachev}),
\begin{equation*}
    \cL_nf(x)=f'(x)b_n(x)+\int_{\R^2}(f(x+y)-f(x)-f'(x)y)M(dy,x).
\end{equation*}
Now let $f \in C_b(\R)$ and observe that, by the change of variables formula and Inequality \eqref{eq:sublingr}, we have
\begin{equation*}
\int_{\R^2}\frac{|y|^2}{1+|y|^2}f(y)M(dy,x)=\int_{\R^2}\frac{|g_n(s,u,x)|^2}{1+|g_n(s,u,x)|^2}f(g_n(s,u,x))\nu_1(dsdu)\le C\left\| f \right\|_{\infty},
\end{equation*}
for some constant $C>0$. Thus we are under the hypotheses of \cite[Theorem 2.2]{Stroock} and we know that the martingale problem associated with $\cL_n$ admits a solution, which guarantees the existence of at least a weak solution of \eqref{eq:V1auxn} through \cite[Theorem 2.3]{Kurtz}. Now we need to show that strong solutions of \eqref{eq:V1auxn} are pathwise unique. To do this, let us first observe that for any $R>0$, $\delta_0 \in (0,1)$ and any $x,y \in \R$ such that $|x|,|y| \le R$ and $|x-y| \le \delta_0$ it holds
\begin{equation}\label{intcontr}
    \int_{\R^2}|g_n(s,u,x)-g_n(s,u,y)|^2\nu_1(dsdu)\le C_R |x-y|,
\end{equation}
where $C_R>0$ is a suitable constant independent of $\delta_0$.
Now consider a sequence $\{a_k\}_{k \in \N} \subset (0,1]$ such that $a_0=1$, $a_k<a_{k-1}$, $a_k \to 0$ as $k \to +\infty$ and $$\int_{a_k}^{a_{k-1}}\frac{1}{r}dr=\log\left(\frac{a_{k-1}}{a_k}\right)=k.$$
For such a sequence $\{a_k\}_{k \in \N}$, there exists a sequence of continuous functions $\rho_k:\R^+ \to \R$ such that ${\rm supp}(\rho_k) \subset (a_k,a_{k-1})$, $0 \le \rho_k(r) \le \frac{2}{kr}$ and
\begin{equation*}
    \int_{a_k}^{a_{k-1}}\rho_k(r)dr=1.
\end{equation*}
Let us further define the sequence of functions $\phi_k: \R \to \R$
\begin{equation*}
    \phi_k(r)=\int_0^{|r|}\int_0^v \rho_k(w)dwdv.
\end{equation*}
One can easily check that $\phi_k \in C^2(\R)$, ${\rm supp}(\phi_k) \subset \R \setminus (-a_k,a_k)$, $\phi'_k(r) \ge 0$ for any $r \ge 0$ and, for any $r \in \R$, $\phi'_k(r) \uparrow 1$ and $\phi_k(r) \uparrow |r|$ as $k \to +\infty$.
Now let us suppose we have two strong solutions $V_n(t)$ and $\widetilde{V}_n(t)$ of Equation \eqref{eq:V1auxn} and let
\begin{align*}
    \Delta(t)&=V_n(t)-\widetilde{V}_n(t)\\
    \ell^b_n(x)&=-\kappa_vx\\
    \ell^g_n(s,u,x,y)&=g_n(s,u,x)-g_n(s,u,y).
\end{align*} 
Then it is clear that
\begin{equation*}
    \Delta(t)=\int_0^t \ell^b_n(\Delta(\tau-))d\tau+\int_0^t\int_{\R^2} \ell^g_n(s,u,V_n(\tau-),\wV_n(\tau-)\widetilde{N}_{\Gamma_2}(d\tau,dsdu).
\end{equation*}
By the It\^o formula for L\'evy-It\^o processes (see \cite[Theorem 4.4.7]{Applebaum}) we have
\begin{align*}
\phi_k(|\Delta(t)|)&=\int_0^t \frac{\phi_k'(|\Delta(\tau-)|)}{|\Delta(\tau-)|}\Delta(\tau-) \ell^b_n(\Delta(\tau-))d\tau\\
&+\int_0^t\int_{\R^2} (\phi_k(|\Delta(\tau-)+\ell^g_n(s,u,V_n(\tau-),\wV_n(\tau-))|)-\phi_k(|\Delta(\tau-)|))\widetilde{N}_{\Gamma_2}(d\tau,dsdu)\\
&+\int_0^t\int_{\R^2} \left(\phi_k(|\Delta(\tau-)+\ell^g_n(s,u,V_n(\tau-),\wV_n(\tau-))|)-\phi_k(|\Delta(\tau-)|)\right.\\
&\left.-\frac{\phi'_k(|\Delta(\tau-)|)}{|\Delta(\tau-)|}\Delta(\tau-)\ell^g_n(s,u,V_n(\tau-),\wV_n(\tau-))\right)\nu_1(dsdu)d\tau.
\end{align*}
Fix any $\delta_0 \in (0,1)$ and $R>0$ and define
\begin{align*}
    T_{\delta_0}&:=\inf\{t \ge 0: \ |\Delta(t)| \ge \delta_0\},\\
    \tau_{R}&:=\inf\{t \ge 0: \ |V_n(t)| \vee |\widetilde{V}_n(t)| \ge R\}.
\end{align*}
Taking the expectation and using the optional stopping theorem we get
\begin{align}\label{eq:mean1}
\begin{split}
&\E[\phi_k(|\Delta(t \wedge T_{\delta_0} \wedge \tau_R)|)=\E\left[\int_0^{t \wedge T_{\delta_0} \wedge \tau_R} \frac{\phi_k'(|\Delta(\tau-)|)}{|\Delta(\tau-)|}\Delta(\tau-) \ell^b_n(\Delta(\tau-))d\tau\right]\\
&\qquad +\E\left[\int_0^{t \wedge T_{\delta_0} \wedge \tau_R}\int_{\R^2} \left(\phi_k(|\Delta(\tau-)+\ell^g_n(V_n(\tau-),\wV_n(\tau-))|)-\phi_k(|\Delta(\tau-)|)\right.\right.\\
&\qquad \left.\left.-\frac{\phi'_k(|\Delta(\tau-)|)}{|\Delta(\tau-)|}\Delta(\tau-)\ell^g_n(s,u,V_n(\tau-),\wV_n(\tau-))\right)\nu_1(dsdu)d\tau\right].
\end{split}
\end{align}
Now let us recall that $\psi_k'(|\Delta_{\tau-}|) \ge 0$ while
\begin{equation*}
    \Delta(\tau-)\ell^b_n(\Delta(\tau-))=-\kappa_v(\Delta(\tau-))^2 \le 0,
\end{equation*}
hence
\begin{equation}\label{eq:est1}
    \E\left[\int_0^{t \wedge T_{\delta_0} \wedge \tau_R} \frac{\phi_k'(|\Delta(\tau-)|)}{|\Delta(\tau-)|}\Delta(\tau-) \ell^b_n(\Delta(\tau-))d\tau\right]\le 0.
\end{equation}
On the other hand, by Taylor's formula with integral remainder and recalling that $\phi''_k(|r|)=\rho_k(|r|)$, it holds
\begin{align*}
    &\phi_k(|\Delta(\tau-)+\ell^g_n(s,u,V_n(\tau-),\wV_n(\tau-))|)-\phi_k(|\Delta(\tau-)|)\\
    &\qquad -\frac{\phi'_k(|\Delta(\tau-)|)}{|\Delta(\tau-)|}\Delta(\tau-)\ell^g_n(s,u,V_n(\tau-))\\
    %&=(\ell^g_n(s,u,V_n(\tau-),\wV_n(\tau-)))^2\\
    %&\qquad \times \int_0^1 \phi_k''(|\Delta(\tau-)+h\ell^g_n(s,u,V_n(\tau-),\wV_n(\tau-))|)(1-h)dh\\
    &=(\ell^g_n(s,u,V_n(\tau-),\wV_n(\tau-)))^2\\
    &\qquad \times \int_0^1 \rho_k(|\Delta(\tau-)+h\ell^g_n(s,u,V_n(\tau-),\wV_n(\tau-))|)(1-h)dh\\
    &\le \frac{2}{k}(\ell^g_n(s,u,V_n(\tau-),\wV_n(\tau-)))^2\\
    &\qquad \times \int_0^1 |\Delta(\tau-)+h\ell^g_n(s,u,V_n(\tau-),\wV_n(\tau-))|^{-1}(1-h)dh\\
    &\le \frac{2}{k}|\Delta(\tau-)|^{-1}(\ell^g_n(s,u,V_n(\tau-),\wV_n(\tau-)))^2.
\end{align*}
Taking the integral in $\nu_1(dsdu)$ and recalling that $\tau \le t \wedge T_{\delta_0} \wedge \tau_R$ we have
\begin{align}\label{eq:est2}
\begin{split}
    &\int_{\R^2}\left(\phi_k(|\Delta(\tau-)+\ell^g_n(s,u,V_n(\tau-),\wV_n(\tau-)))-\phi_k(|\Delta(\tau-)|)\right.\\
    &\qquad \left.-\frac{\phi'_k(|\Delta(\tau-)|)}{|\Delta(\tau-)|}\Delta(\tau-)\ell^g_n(s,u,V_n(\tau-),\wV_n(\tau-))\right)\nu_1(dsdu)\\
    &\le \frac{2}{k}|\Delta(\tau-)|^{-1}\int_{\R^2}(\ell^g_n(s,u,V_n(\tau-),\wV_n(\tau-)))^2\nu_1(dsdu)\le C_R\frac{2}{k},
    \end{split}
\end{align}
where we also used Inequality \eqref{intcontr}. Thus, applying Inequalities \eqref{eq:est1} and \eqref{eq:est2} to \eqref{eq:mean1}, we get
\begin{align*}
\E[\phi_k(|\Delta_{t \wedge T_{\delta_0} \wedge \tau_R}|)]\le C_R\frac{2}{k} \E[t \wedge T_{\delta_0}\wedge \tau_R].
\end{align*}
Taking the limit as $k \to +\infty$, recalling that $\psi_k(r) \uparrow |r|$, by the monotone convergence theorem we achieve 
\begin{align*}
\E[|\Delta_{t \wedge T_{\delta_0} \wedge \tau_R}|]=0.
\end{align*}
Now let us take $R \to +\infty$, recalling that $\tau_R \uparrow +\infty$ since both $V_n(t)$ and $\widetilde{V}_n(t)$ are non-explosive, so that
\begin{align*}
0=\E[|\Delta_{t \wedge T_{\delta_0}}|]\ge \delta_0 \bQ(T_{\delta_0} \le t)
\end{align*}
and thus $\bQ(T_{\delta_0} \le t)=0$. Being both $t>0$ and $\delta_0 \in (0,1)$ arbitrary, we achieve the desired pathwise uniqueness. By \cite[Theorem 2]{Barczy}, this guarantees the existence of a strong solution $V_n(t)$ to \eqref{eq:V1auxn}. Now let us define $\tau_n:=\inf \{t \ge 0: \ |V_n(t)| \ge n\}$. By definition of $\psi_n^{(1)}$ and by the pathwise uniqueness of \eqref{eq:V1auxn}, it is clear that $\tau_n$ is the first exit time from $(-n,n)$ of any $V_m(t)$ such that $m \ge n$.
Hence we can define the following process
\begin{equation*}
    V(t):=V_n(t), \ t \le \tau_n.
\end{equation*}
Again, by definition of $b_n$ and $g_n$, such a process solves \eqref{eq:V1aux} up to $\tau_n$ for any $n \in \N$. Let us stress that we can use the same arguments as before to guarantee pathwise uniqueness for the strong solutions of Equation \eqref{eq:V1aux} and $\tau_n=\inf\{t \ge 0: \ |V(t)| \ge n\}$. However, we already proved that the solutions of Equation \eqref{eq:V1aux} are non-explosive, thus $\tau_n \uparrow +\infty$ as $n \to +\infty$, implying that $V(t)$ is the pathwise unique global strong solution of \eqref{eq:V1aux}. Now let $\tau^V_0:=\inf\{t \ge : \ V(t) < 0\}$ and observe that, for $t \le \tau^V_0$, we can rewrite \eqref{eq:V1aux} as \eqref{eq:V1}. Recalling that $V(t)$ is a c\'adl\'ag process, it is clear that $\bQ(\tau_0>0)=1$. Hence $V(t)$ is a strong solution of \eqref{eq:V1} up to $\tau_0$. Finally, to guarantee pathwise uniqueness of $\eqref{eq:V1}$, let us consider any other strong solution $\widetilde{V}(t)$ of \eqref{eq:V1} up to a stoppin time $\widetilde{\tau}_0$ with $\bQ(\widetilde{\tau}_0>0)=1$, let us consider the global solution $V(t)$ of \eqref{eq:V1aux}. Then, both $\widetilde{V}(t)$ and $V(t)$ are solutions of \eqref{eq:V1aux} up to $\widetilde{\tau}_0$ and thus, by pathwise uniqueness, they coincide. This also implies that $\widetilde{\tau}_0\le \tau^V_0$, concluding the proof.
%Hence $V(t)$ is the pathwise unique non-negative strong solution of \eqref{eq:V1} up to time $\tau^V_0$.
\end{proof}

%%%%%%%%%%%%%%

\section{Proof of Theorem \ref{thm:exunrf0}}

\begin{proof} \label{proof:exunrf0}
If $\rho_{sf}=0$, then existence and uniqueness is proven analogously as in Theorem \ref{thm:exunV0rd0}.\\

Let us consider the case $\rho_{sf} \not =0$. As in Theorem \ref{thm:exunV0rd0}, let us find the L\'evy-It\^o decomposition of $(\gamma(t),\wGamma_1(t),\wGamma_2(t))$. Indeed, if we consider its Poisson random measure, defined for $t \ge 0$ and $A  \in \cB(\R^3 \setminus \{0\})$ as
\begin{equation*}
    N_{\wGamma}(t,A)=\sum_{0 \le s \le t}1_A((\Delta \gamma(s), \Delta \wGamma_1(s), \Delta \wGamma_2(s))),
\end{equation*} 
and its compensated version $\wN_{\wGamma}(dt,dsdu)=N_{\wGamma}(dt,dsdu)-\nu_2(dsdu)dt$, where $\nu_2(dsdu)$ is the L\'evy measure of $(\gamma(t),\wGamma_1(t),\wGamma_2(t))$ identified by Lemma \ref{lem:Levy2}, it holds
\begin{align*}
    (\gamma(t),\wGamma^{(1)}(t),\wGamma^{(2)}(t))=\int_{\R^3}(s,u_1,u_2)\wN_{\wGamma}(t,dsdu)
\end{align*}
and then we can rewrite Equations \eqref{eq:V0R} and \eqref{eq:rf0R} as
\begin{align}
    dV(t)&=\kappa_v(a_v-V(t-))dt+\int_{\R^3}(\theta_vs+\sigma_vu_1 \sqrt{V(t-)})\wN_{\wGamma}(dt,dsdu), & V_0>0 \label{eq:V0R1}\\
    \begin{split}
    dr^f(t)&=(\kappa_f(a_f-r^f(t-))-\sigma_f \rho_{sf} \sqrt{V(t-)r^f(t-)})dt\\
    &+\int_{\R^3}(\theta_fs+\sigma_f\rho_{vf}u_1 \sqrt{r^f(t-)}+\sigma_f\sqrt{1-\rho^2_{vf}}u_2\sqrt{r^f(t-)})\wN_{\wGamma}(dt,dsdu)\end{split} & r^f_0>0. \label{eq:rf0R1}
\end{align}
Now let $b_V$ and $g_V$ be as in the proof of Theorem \ref{thm:exunV0rd0} and define, for $s \in \R$ and $u,x \in \R^2$,
\begin{align*}
    b_{r^f}(x)&=\kappa_f(a_f-x_2)-\sigma_f \rho_{sf} \sqrt{x_1x_2}1_{\R^+}(x_1 \wedge x_2),\\
    g_{r^f}(s,u,x_2)&=(\theta_fs1_{\R^+}(s)+\sigma_f\rho_{vf}u_1\sqrt{x_2}1_{\R^+}(x_2)+\sigma_f\sqrt{1-\rho^2_{vf}}u_2\sqrt{x_2}1_{\R^+}(x_2)),\\
    b_{V,r^f}(x)&=(b_V(x_1) ,b_{r^f}(x)), \qquad g_{V,r^f}(s,u,x)= (g_V(s,u_1,x_1), g_{r^f}(s,u,x_2)).
\end{align*}
Consider the auxiliary equations
\begin{align}
    dV(t)&=b_V(V(t-))dt+\int_{\R^3}g_V(s,u_1,V(t-))\wN_{\wGamma}(dt,dsdu), & V(0)=V_0>0 \label{eq:V0R11}\\
    \begin{split}
    dr^f(t)&=b_{r^f}(V(t-),r^f(t-))dt\\
    &+\int_{\R^3}g_{r^f}(s,u,r^f(t-))\wN_{\wGamma}(dt,dsdu)\end{split} & r^f(0)=r^f_0>0. \label{eq:rf0R11}
\end{align}
As in Theorem \ref{thm:exunV0rd0}, one can prove by direct evaluation (and by means of Young's inequality with exponents $p=\frac{4}{3}$ and $q=4$ to handle $x_2\sqrt{x_1x_2}$) that for any $x \in \R^2$ it holds
    
    \begin{equation*}
    2\langle x, b_{V,r^f}(x)\rangle+\int_{\R^3}|g_{V,r^f}(s,u,x)|^2\nu_2(dsdu)\le C(1+|x|^2),
    \end{equation*}
    where $C>0$ is a suitable constant, which guarantees, by means of \cite[Theorem 2.2]{Xi}, that the strong solutions of \eqref{eq:V0R11} and \eqref{eq:rf0R11} (and thus also the strong solutions of \eqref{eq:V0R1} and \eqref{eq:rf0R1}) are non-explosive.\\
    Now we need to distinguish among the two cases $\rho_{sf}>0$ and $\rho_{sf}<0$.\\
    In the case $\rho_{sf}>0$, once we observe that there exists a constant $C>0$ such that for any $x,y \in \R^2$
%    Let us first consider the case $\rho_{sf}>0$. Arguing as in Theorem \ref{thm:exunV0rd0}, it can be shown that
    \begin{equation*}
        \int_{\R^3}|g_{V,r^f}(s,u,x)-g_{V,r^f}(s,u,y)|^2\nu_2(dsdu) \le C|x-y|,
    \end{equation*}
    we can proceed as in Theorem \ref{thm:exunV0rd0} by defining for any $n \in \N$
    \begin{align*}
     b_n(x)&=(\psi^{(1)}_n(x_1)b_{V}(x_1), \psi^{(2)}_n(x)b_{r^f}(x))=:(b_n^{(1)}(x_1),b_n^{(2)}(x))\\
     g_n(s,u,x)&=(\psi_n^{(1)}(x_1)g_V(s,u_1,x_1), \ \psi_n^{(1)}(x_2)g_{r^f}(s,u,x_2))=:(g_n^{(1)}(s,u_1,x_1),g_n^{(2)}(s,u,x_2))
    \end{align*}
    where for any $x \in \R^2$ it holds $\psi_n^{(2)}(x)=\psi_n^{(1)}(x_1)\psi_n^{(1)}(x_2)$. The only real difference concerns the proof of the pathwise uniqueness for Equations
    \begin{align}
        dV_n(t)&=b_n^{(1)}(V_n(t-))dt+\int_{\R^3}g_n^{(1)}(s,u_1,V_n(t-))\wN_{\wGamma}(dt,dsdu), \qquad V_n(0)=V_0\label{eq:V0n2}\\
        dr^f_n(t)&=b_n^{(2)}(V_n(t-),r_n^f(t-))dt+\int_{\R^3}g_V^{(2)}(s,u,r_n^f(t-))\wN_{\wGamma}(dt,dsdu), \qquad r^f_n(0)=r^f_0.\label{eq:rf0n2}
    \end{align}
    Indeed, once we consider two pairs $(V_n(t),r_n^f(t))$ and $(\wV_n(t),\widetilde{r}_n^f(t))$ of strong solutions of Equations \eqref{eq:V0n2} and \eqref{eq:rf0n2} and we define 
    \begin{equation*}
    \Delta(t):=(V_n(t)-\wV_n(t),r_n^f(t)-\widetilde{r}_n^f(t))=:(\Delta^{(1)}(t),\Delta^{(2)}(t)),
    \end{equation*}
    we can observe that Equation \eqref{eq:V0n2} coincides with Equation \eqref{eq:V1auxn} and then $\Delta^{(1)}\equiv 0$ almost surely by the pathwise uniqueness proved in Theorem \ref{thm:exunV0rd0}. Thus, being $V_n(t)=\wV_n(t)$ almost surely, one can easily check that, defining
    \begin{equation*}
    \ell_n^b(x,y)=b_n(x)-b_n(y)
    \end{equation*}
        it holds
    \begin{align*}
        \langle &\Delta(t), \ell_n^b(V_n(t),r_n^f(t),\wV_n(t),\widetilde{r}_n^f(t))\rangle\\
        &\qquad \qquad =-\kappa_f(\Delta^{(2)}(t))^2-\sigma_f\rho_{sf}\sqrt{V_n(t)}\Delta^{(2)}(t)(\sqrt{r_n^f(t)}-\sqrt{\widetilde{r}_n^f(t)}) \le 0.
    \end{align*}
    Once this is clear, the same arguments as in Theorem \ref{thm:exunV0rd0} conclude the proof in the case $\rho_{sf}>0$.\\
Now let us handle the case $\rho_{sf}<0$. Fix $n_0 \in \N$ big enough to have $V_0,r^f_0 \in \left[\frac{1}{n_0},n_0\right]$. For any $n \ge n_0$ consider a function $\psi^{(3)}_n \in C_c^\infty(\R)$ such that $0 \le \psi^{(3)}_n(x) \le 1$ for any $x \in \R$, $\psi^{(3)}_n(x)=1$ for any $x \in \left[\frac{1}{n},n\right]$ and ${\rm supp}\psi^{(3)}_n \subset \left(\frac{1}{n+1},n+1\right)$. Let us also denote $\psi^{(4)}_n(x_1,x_2)=\psi^{(3)}_n(x_1)\psi^{(3)}_n(x_2)$ for any $(x_1,x_2) \in \R^2$. \\

This time we define
\begin{align*}
     b_n(x)&=(b_V(x_1), \kappa_f(a_f-x_2)-\sigma_f\rho_{sf}\psi_n^{(4)}(x)\sqrt{x_1x_2})=(b^{(1)}_n(x_1),b^{(2)}_n(x))\\
     g_n(s,u,x)&=(\theta_v s+\sigma_v u_1\psi_n^{(3)}(x_1)\sqrt{x_1}, \\
     &\qquad \qquad \qquad \theta_f s+\sigma_f\rho_{vf} u_1\psi_n^{(3)}(x_2)\sqrt{x_2}+\sigma_f \sqrt{1-\rho^2_{vf}}u_2\psi_n^{(3)}(x_2)\sqrt{x_2})\\    &=(g^{(1)}_n(s,u_1,x_1),g^{(2)}_n(s,u,x_2))
\end{align*}
and the auxiliary SDEs
\begin{align}
    dV_n(t)&=b_n^{(1)}(V_n(t-))dt+\int_{\R^3}g^{(1)}_n(s,u_1,V_n(t-))\wN_{\wGamma}(dt,dsdu_1du_2), \quad V_n(0)=V_0 \label{eq:V0R12n}\\
    dr^f_n(t)&=b_n^{(2)}(V_n(t-),r_n^f(t-))dt+\int_{\R^3}g^{(2)}_n(s,u_1,u_2,r^f_n(t-))\wN_{\wGamma}(dt,dsdu_1du_2), \quad r^f_n(0)=r^f_0 \label{eq:rf0R12n}.
\end{align}
By direct evaluation, one can check that there exists a constant $C>0$ such that for any $x \in \R^2$ it holds.
\begin{equation*}
    |b_n(x)|^2+\int_{\R^3}|g_n(s,u,x)|^2\nu_2(dsdu) \le C(1+|x|^2).
\end{equation*}
Moreover, by definition, it is clear that $b_n$ is Lipschitz (since it belongs to $C^1(\R^2)$ and its gradient is bounded by the definition of $\psi^{(4)}(x)$). For the same reason, for any $s \in \R$ and $u \in \R^2$, also $g_n(s,u,\cdot)$ is a a Lipschitz function and one can prove, by direct evaluation of the supremum of the gradient, that there exists a constant $C>0$ such that for any $s \in \R$ and $u,x,y \in \R^2$ it holds
\begin{equation*}
    |g_n(s,u,x)-g_n(s,u,y)|^2 \le C|u|^2|x-y|^2.
\end{equation*}
Hence, we know that there exists a constant $C>0$ such that for any $x,y \in \R^2$ it holds
\begin{equation*}
    |b_n(x)-b_n(y)|^2+\int_{\R^3}|g_n(s,u,x)-g_n(s,u,y)|^2\nu_2(dsdu) \le C|x-y|^2.
\end{equation*}
Hence we are under the classical Lipschitz and sublinear growth hypotheses and by \cite[Theorem 6.2.3]{Applebaum} there exists a pathwise unique strong solution $(V_n(t),r^f_n(t))$ to Equations \eqref{eq:V0R12n} and \eqref{eq:rf0R12n}. Now let us define the sequence of Markov times 
\begin{equation*}
    \tau_n:=\inf\left\{t \ge 0: \ (V_n(t),r^f_n(t))\not\in \left[\frac{1}{n},n\right]\right\}
\end{equation*}
and observe that, as $t<\tau_n$, $V_n(t)$ solves Equation \eqref{eq:V0R1} and thus it coincides with $V(t)$ by pathwise uniqueness. On the other hand, let us also observe that for $m>n$ $r^f_n(t)$ and $r^f_m(t)$ solve the same equation up to $\tau_n \wedge \tau_n^m$, where
\begin{equation*}
    \tau_n^m:=\inf\left\{t \ge 0: \ (V_m(t),r^f_m(t))\not\in \left[\frac{1}{n},n\right]\right\}.
\end{equation*}
By pathwise uniqueness, we know that $r^f_n(t)=r^f_m(t)$ up to $\tau_n \wedge \tau_n^m$ and thus, as a consequence, $\tau_n=\tau_n^m$. This implies that $\tau_n$ is an increasing sequence of Markov times and then we can define the Markov time $\tau_0^f:=\lim_n \tau_n$. Moreover, for any $t \in [0,\tau_0^f)$, one can define
\begin{equation*}
    r^f(t):=r^f_n(t), \qquad t<\tau_n.
\end{equation*}
However, since $V_n(t)=V(t)$ for $t < \tau_n$, $r^f(t)=r^f_n(t)$ actually solves Equation \eqref{eq:rf0R1}. Moreover, if $\widetilde{r}^f(t)$ is another strong solution of \eqref{eq:rf0R1} and we consider
\begin{equation*}
    \widetilde{\tau}_n:=\inf\left\{t \ge 0: \ (V(t),\widetilde{r}^f(t))\not\in \left[\frac{1}{n},n\right]\right\},
\end{equation*}
then $\widetilde{r}^f(t)$ solves also \eqref{eq:rf0R12n} up to $\widetilde{\tau}_n$. This implies that $\widetilde{r}^f(t)=r_n^f(t)$ up to $\widetilde{\tau}_n$ and thus $\widetilde{\tau}_n=\tau_n$. Thus, in particular, $\widetilde{r}^f(t)=r^f(t)$ and we have that $r^f(t)$ is the unique pathwise solution of \eqref{eq:rf0R1} up to $\tau_0^f$.
\end{proof}

%%%%%%%%%%%%%%

\section{Proof of Theorem \ref{thm:QuasiGyongy}}

\begin{proof} \label{proof:QuasiGyongy}
Let us observe that for any $t \in [0,T]$ and any $v \in [0,t]$ it holds
\begin{align*}
    X^n(v)-X^n_m(v)&=\int_0^{v}(b_n(X^n(z-))-b_n(X^n_m(\eta_{m}(z)-)))dz\\
    &+\int_0^v\int_{\R^5}(g_n(s,u,X^n(z-))-g_n(s,u,X^n_m(\eta_m(z)-)))\wN_{\oGamma}(dz,dsdu).
\end{align*}
By the It\^o formula for L\'evy-It\^o processes and some simple algebraic manipulations we have
\begin{align*}
    |&X^n(v)-X^n_m(v)|^2=2\int_0^{v}\langle X^n(z-)-X^n_m(z-),b_n(X^n(z-))-b_n(X^n_m(z-))\rangle dz\\
    &+2\int_0^{v}\langle X^n(z-)-X^n_m(z-),b_n(X^n_m(z-))-b_n(X^n_m(\eta_{m}(z)-))\rangle dz\\
    &+\int_0^v\int_{\R^5}a_1(s,u,z)\wN_{\oGamma}(dz,dsdu)\\
    &+2\int_0^v\int_{\R^5}|g_n(s,u,X^n(z-))-g_n(s,u,X^n_m(z-))|^2\nu_4(dsdu)dz\\
    &+2\int_0^v\int_{\R^5}|g_n(s,u,X^n_m(z-))-g_n(s,u,X^n_m(\eta_m(z)-))|^2\nu_4(dsdu)dz,
\end{align*}
where we set
\begin{align*}
    a_1(s,u,z)&:=(|g_n(s,u,X^n(z-))-g_n(s,u,X^n_m(\eta_m(z)-))|^2\\
    &-2\langle X^n(z-)-X^n_m(z-),g_n(s,u,X^n(z-))-g_n(s,u,X^n_m(\eta_m(z)-))\rangle).
\end{align*}
Let us observe that, by the fact that $b_n$ is Lipschitz and bounded and by Lemma \ref{lem:Lip}, it holds
\begin{align*}
    |&X^n(v)-X^n_m(v)|^2\le  C\int_0^t |X^n(z-)-X^n_m(z-)|^2dz\\
    &\quad +C\int_0^{t}|X^n_m(z-))-X^n_m(\eta_{m}(z)-)|^2 dz+\int_0^v\int_{\R^5}a_1(s,u,z)\wN_{\oGamma}(dz,dsdu),
\end{align*}
where we also used Cauchy-Schwartz and Young's inequality. Now let $p \ge 2$ and observe that, by applying Jensen's inequality, taking the supremum and then the expectation and finally applying Kunita's first inequality and Lemma \ref{lem:Lip}, it holds
\begin{align}\label{eq:estEM2}
\begin{split}
    \E&\left[\sup_{v \in [0,t]}|X^n(v)-X^n_m(v)|^{2p}\right]\le C\int_0^t \E\left[\sup_{0 \le w \le z}|X^n(w)-X^n_m(w)|^{2p}\right]dz\\
    &+C\int_0^{t}\E[|X^n_m(z-)-X^n_m(\eta_{m}(z)-)|^{2p}] dz.
\end{split}
\end{align}
Now let us estimate $\E[|X^n_m(z-)-X^n_m(\eta_{m}(z)-)|^{2p}]$. By stochastic continuity of $X^n_m(t)$, we know that $X^n_m(z)=X^n_m(z-)$ and $X^n_m(\eta_{m}(z)-)=X^n_m(\eta_{m}(z))$ almost surely. Moreover, we have
\begin{align*}
    X^n_m(z-)-X^n_m(\eta_{m}(z))&=\int_{\eta_m(z)}^zb_n(X_m^n(\eta_m(v)-))dv\\
    &+\int_{\eta_m(z)}^z\int_{\R^5}g_n(s,u,X_m^n(\eta_m(v)-))\wN(dv,dsdu)
\end{align*}
and then, by using the fact that $b_n$ is bounded, Lemma \ref{lem:Lip} and Jensen's inequality, we achieve
\begin{align*}
    \E&\left[|X^n_m(z)-X^n_m(\eta_{m}(z))|^{2p}\right]\\
    &\le Cm^{-4p}+C\E\left[\sup_{w \in [\eta_m(z),z]}\left|\int_{\eta_m(z)}^w\int_{\R^5}g_n(s,u,X_m^n(\eta_m(v)-))\wN(dv,dsdu)\right|^{2p}\right].
\end{align*}
Again, by Kunita's first inequality, Lemma \ref{lem:Lip} and Jensen's inequality holds
\begin{align*}
    \E&\left[|X^n_m(z)-X^n_m(\eta_{m}(z))|^{2p}\right]\le Cm^{-4p}+Cm^{-2p}+Cm^{-2}\le Cm^{-2}.
\end{align*}
Using the latter inequality into Inequality \eqref{eq:estEM2} we have
\begin{align*}
\begin{split}
    \E&\left[\sup_{v \in [0,t]}|X^n(v)-X^n_m(v)|^{2p}\right]\le C\int_0^t \E\left[\sup_{0 \le w \le z}|X^n(w)-X^n_m(w)|^{2p}\right]dz+Cm^{-2}.
\end{split}
\end{align*}
which, in turn, by Gr\"onwall's inequality implies
\begin{equation*}
    \E\left[\sup_{v \in [0,t]}|X^n(v)-X^n_m(v)|^{2p}\right] \le Cm^{-2}.
\end{equation*}
Now consider any $\theta<\frac{1}{4}$ and use Markov's inequality to remark that
\begin{equation*}
    \bQ(\sup_{t \in [0,T]}|X^n(t)-X^n_m(t)|\ge m^{-\theta})\le m^{2p\theta}\E\left[\sup_{t \in [0,T]}|X^n(v)-X^n_m(v)|^{2p}\right] \le C m^{2(p\theta-1)}.
\end{equation*}
Recalling that $\frac{1}{2\theta}>2$, we can always choose $p \ge 2$ so that $2(1-p\theta)>1$. With such a choice, it holds
\begin{equation*}
    \sum_{m=1}^{+\infty}\bQ(\sup_{t \in [0,T]}|X^n(t)-X^n_m(t)|\ge m^{-\theta})\le C \sum_{m=1}^{+\infty}m^{2(p\theta-1)}<\infty.
\end{equation*}
The Borel-Cantelli Lemma concludes the proof.
\end{proof}

\newpage

%\section*{References}


\begin{thebibliography}{40} 

\bibitem{Ahlip}
Ahlip, R. and Rutkowski, M. (2013). Pricing of foreign exchange options under the Heston stochastic volatility model and CIR interest rates. J. Quantitative Finance, 13(6):955-966.

\bibitem{Akgiray}
Akgiray, V. Booth, G. G. (1988). Mixed diffusion-jump process modeling of exchange rate movements. J. Economics and Statistics.

\bibitem{Asmussen}
Asmussen, S., Glynn P. W. (2007). Stochastic simulation: algorithms and analysis. Vol. 57. New York: Springer.

\bibitem{Applebaum} Applebaum, D. (2009). L\'evy processes and stochastic calculus. Cambridge University Press.

\bibitem{Bakshi}1997,
Bakshi, G., Cao, C., Chen, Z. (1997). Empirical performance of alternative option pricing models. J. finance.

\bibitem{Barczy} Barczy, M., Li Z., Pap G. (2015). Yamada-Watanabe results for stochastic differential equations with jumps. International Journal of Stochastic Analysis 2015.
  
\bibitem{Barndorff}
Barndorff-Nielsen, O. E. and Shephard, N. (2001). Non-Gaussian Ornstein-Uhlenbeck based models and some of their uses in financial economics (with discussion). J. Roy. Statist. Soc., Ser. B. 63 167-241.

\bibitem{Barndorff1977}
Barndorff-Nielsen, Ole. (1977). Exponentially decreasing distributions for the logarithm of particle size. J. Mathematical and Physical Sciences.

\bibitem{Barndorff1997}
Barndorff-Nielsen, O. (1997). Normal Inverse Gaussian Distributions and Stochastic Volatility Modelling, Scandinavian. Statistics. 

\bibitem{BaroneAdesi}
Barone-Adesi, G. (2005). The saga of the {A}merican put. J. Banking {\&} Finance.

\bibitem{Bates}
Bates, D. S. (1996). Jumps and stochastic volatility: The exchange rate processes implicit in deutsche mark options, Rev. Financ. Stud. 9, 69-107.   

\bibitem{Barndorff-Nielsen}
Barndorff-Nielsen, O. E.,  Mikosch, T. Resnick, S. I.  (2013). Lévy Processes, Theory and Applications, Birkhauser.

\bibitem{Benhamou}
Benhamou, E., Gobet, E. and Miri, M. (2010). Time dependent Heston model, SIAM J. Financ. Math. 1,  289-325.

\bibitem{Bertoin} Bertoin, J. (1999) Subordinators: examples and applications. Lectures on probability theory and statistics. Springer, Berlin, Heidelberg. 1-91.

\bibitem{Bogachev} Bogachev, V. I. (2007) Measure theory. Vol. 1. Berlin: Springer.

\bibitem{Boyarchenko}
Boyarchenko, S. I. and Levendorskii, S. Z. (2002). Perpetual American options under L\'{e}vy processes. SIAM J. Control Optim.401663–1696

\bibitem{Braunstein}
Braunstein, A. (2008). {A}merican Option Approximations. [Online; accessed 5. Aug. 2019]. 

\bibitem{Broadie}
Broadie, M. Glasserman, P. (1997). Pricing American-style securities using simulation. Economic Dynamics and Control.

\bibitem{Bunch}
Bunch, D.S., Johnson, H. (1992). A Simple and Numerically Efficient Valuation Method for {A}merican Puts Using a Modified Geske-Johnson Approach. J. Finance.


\bibitem{Campa}
Campa, JM., Chang, PHK.,  Reide,r RL. (1998). Implied exchange rate distributions: evidence from OTC option markets. J. international Money and Finance.

\bibitem{Carriere}
Carriere, Jacques F. (1996). Valuation of the early-exercise price for options using simulations and nonparametric regression. Mathematics and Economics.

\bibitem{DEXUSEU}
Board of Governors of the Federal Reserve System (US). (2021). U.S./ Euro Foreign Exchange Rate [DEXUSEU]. \url{https://fred.stlouisfed.org/series/DEXUSEU}. Online; accessed 7 June 2021.
  
\bibitem{Dorbov} 
Dorbov, B. (2015). Monte Carlo Simulation with Machine Learning for Pricing American Options and Convertible Bonds. SSRN Electronic.

\bibitem{EURONTD156N}
 ICE Benchmark Administration Limited. (2021). Overnight London Interbank Offered Rate (LIBOR), based on Euro [EURONTD156N].
  \url{https://fred.stlouisfed.org/series/EURONTD156N} , {Online; accessed 7 June 2021}.
   
\bibitem{Fallah}
 Fallah, L. Najafi, A. R. Mehrdoust, F. (2018). A fractional version of the Cox-Ingersoll-Ross interest rate model and pricing double barrier option with Hurst index $H \in (2/3 , 1)$, Communications in Statistics. Theory and Methods
 
\bibitem{Fama}
 Fama, F. Eugene,. (1965). The behavior of stock-market prices. journal of Business.

\bibitem{Friz} Friz, P. K., Zhang H. (2018). Differential equations driven by rough paths with jumps. Journal of Differential Equations 264.10: 6226-6301.

\bibitem{Fu}
Fu, Z., Li Z. (2010). Stochastic equations of non-negative processes with jumps. Stochastic Processes and their Applications 120.3: 306-330.


\bibitem{Grzelak}
Grzelak, L. A. Oosterlee, C. W. (2011). On the Heston model with stochastic interestrates. SIAM Journal on Financial Mathematics, 2:255-286.

\bibitem{Grzelak1}
Grzelak, L. A.Oosterlee, C. W. (2012). An Equity-Interest Rate hybrid model with Stochastic Volatility and the interest rate smile. Computational Finance.

\bibitem{Grzelak2}
Grzelak, L. A. and C. W. Oosterlee. (2012). On Cross-Currency Models with StochasticVolatility and Correlated Interest Rates. Applied Mathematical Finance.

\bibitem{Grant}
Grant, D. Vora, G. Weeks, D. (1997). Simulation and the Early Exercise Option Problem. J. Financial Engineering.

\bibitem{Gyongy} Gy\"{o}ngy, I. (1998) A note on Euler's approximations. Potential Analysis 8.3: 205-216.


\bibitem{Higham}
 Higham, D. J., (2004). An introduction to financial option valuation. United States of America by Cambridge University Press, New York. 
 
\bibitem{Horn}
Horn, R. A., Johnson, C. R., (2012). Matrix analysis. Cambridge University Press.
 
\bibitem{Higham3}
Higham,D. J. Mao, X. Stuart, A. M., (2003). Exponential mean square stability of numericalsolutions to stochastic differential equations, London Mathematical Society J. Comput. and Math.

\bibitem{Hull}
Hull, J., White, A. (1987). The pricing of options on assets with stochastic volatilities. J. Finance 42 281-300.

\bibitem{Heston}
Heston, S. L. (1993). A closed-form solution for options with stochastic volatility with applications to bond and currency options, Rev. Financ. Stud. 6, 327-343.

\bibitem{Huang}
Huang, J. Subrahmanyam, M. G., George, Y. G. (1996). Pricing and Hedging {A}merican Options: A Recursive Integration Method. J. Financial Studies.

\bibitem{Kavacs}
Kavacs, B., (2012). American option pricing with LSM algorithm and analytic bias correction, 12-13.

\bibitem{Kienitz}
Kienitz, J., Wetteran, D., (2012). Financial modelling, Springer.

\bibitem{Kim}
Kim, B. J.,   Ma, YK., Choe, HJ. (2013). A Simple Numerical Method for Pricing an {A}merican Put Option. J. Appl. Math.

\bibitem{KimSBM} Kim, P., Song ,R., Vondraček, Z. (2012). Potential theory of subordinate Brownian motions revisited. Stochastic analysis and applications to finance: Essays in honour of Jia-an Yan. 243-290.

\bibitem{Kou1}
Kou, S. G.  (2002). A Jump-Diffusion Model for Option Pricing. Management Science, 48, 1086-1101.

\bibitem{Kou}
Kou, S. G. (2008). Jump Diffusion Models for Asset Pricing in Financial Engineering. In Handbooks in OR and MS, Vol, 15, Ch. 2, edited by J. Birge and V. Linetsky, Elsevier.

\bibitem{Kurtz} Kurtz, T. G. (2011) Equivalence of stochastic equations and martingale problems, Stochastic analysis 2010. Springer, Berlin, Heidelberg: 113-130.

\bibitem{Longstaff}
 Longstaff, F. A.,Schwartz, E. S., (2001). Valuing American options by simulation: a simple least-squares approach, Rev. Financ. Stud. 14 (1). 113-147. 

\bibitem{Madan}
Madan, D. B., Seneta, E., (1990). The V. G. Model for Share Market Returns. J. Business. 63, 511-524.

\bibitem{Madan1}
Madan, D. B., Carr, P., Chang, E. C., (1998) The Variance Gamma Process and Option Pricing. Review of Finance, 79–105.

\bibitem{Madan2}
Madan. D. B. Seneta. E., (1987).  Chebyshev  Polynomial  Approximations and Characteristic 
Function Estimation. J. Royal Statistical Society B, 49: 163-169.

\bibitem{Margrabe}
Margrabe, W. (1978). The value of an option to exchange one asset for another. J. Finance.

\bibitem{Merton}
Merton, R. C. (1976). Option pricing when underlying stock returns are discontinuous. J. of Financial Economics. Vol. 3, 125-144.

\bibitem{Mikhailov}
Mikhailov, S. and Nogel, U. (2003). Heston's stochastic volatility model: Implementation, calibration and some extensions, Wilmott J. 7, 74-79.

\bibitem{Muthuraman}
Muthuraman, K., Kuma, S. (2008). Solving Free-boundary Problems with Applications in Finance. J. Found. Trends. Stoch. Sys.

\bibitem{Oksendal}
Oksendal, B. and Sulem,A. (2004). Applied Stochastic Control of Jump Diffusions.

\bibitem{Orlando}
Orlando, G. Mininni, R. Bufalo, M. (2018). A New Approach to CIR Short Term Rates Modelling. New Methods in Fixed Income Modeling.

\bibitem{Orlando2}
Orlando, G. Mininni, R. Bufalo, M. (2019). Interest rates calibration with a CIR model. Journal of Risk Finance.

\bibitem{Orlando3}
Orlando, G. Mininni, R. Bufalo, M. (2019). A New Approach to Forecast Market Interest Rates Through the CIR Model. Studies in Economics and Finance.

\bibitem{Orlando4}
Orlando, G. Mininni, R. Bufalo, M. (2019). Forecasting interest rates through Vasicek and CIR models: a partitioning approach. Journal of Forecasting.

\bibitem{Owen}
Owen, J. Ramon, R.  (1983). On the class of elliptical distributions and their applications to the theory of portfolio choice. The Journal of Finance.

\bibitem{Oksendal1}
Nunno, D, G. Oksendal, and B. Proske, F. (2008). Malliavin Calculus for L\'{e}vy Processes with Applications to Finance.

\bibitem{Robert}
Robert, A., van de Geijn, (2011). Notes on Cholesky Factorization. The University of Texas, Austin.

\bibitem{Rubinstein1}
Rubinstein,M. (1978). Nonparametric tests of alternative option pricing models using all reported trades and quotes on the 30 most active CBOE option classes from August 23, 1976 through August 31, 1978. J. Finance.
  
 \bibitem{Rubinstein2}1994,
 Rubinstein,M. (1994). Implied binomial trees. J. finance.

\bibitem{Samimi}
 Samimi, O., Mardani, Z., Sharafpour, S., Mehrdoust, F. (2016). LSM Algorithm for Pricing American Option Under Heston-Hull-White’s Stochastic Volatility Model. Computational Economics.

\bibitem{Sato} Sato, K.-I. (1999). L\'evy processes and infinitely divisible distributions. Cambridge University Press.

\bibitem{Steven} 
Steven, E. (2000). Shreve, Stochastic Calculus for Finance, Springer finance series.

\bibitem{Stephens}
Stephens, M.A., (1974). EDF Statistics for Goodness of Fit and Some Comparisons. J. 
American Statistical Association, 69: 730-737

\bibitem{Stein}
Stein, E., \& Stein, J. (1991). Stock Price Distributions with Stochastic Volatility: an Analytic Approach. Review of financial studies.

\bibitem{Stroock} Stroock, D. W. (1975). Diffusion processes associated with Lévy generators. Zeitschrift für Wahrscheinlichkeitstheorie und verwandte Gebiete 32.3: 209-244.

\bibitem{Teng}
Teng, L., Ehrhardt, M. and Gunther, M. (2014). The dynamic Correlation model and its application to the Heston model, Preprint 14/09, University of Wuppertal.

\bibitem{Tsitsiklis}
Tsitsiklis, J. N. Fellow, Van, R. B . (2001). Regression Methods for Pricing Complex American-Style Options. IEEE.

\bibitem{USDONTD156N}
 ICE Benchmark Administration Limited (IBA).(2021). Overnight London Interbank Offered Rate (LIBOR), based on U.S. Dollar [USDONTD156N]. \url{https://fred.stlouisfed.org/series/USDONTD156N}. Online; accessed 7 June 2021.

\bibitem{Teng}
Teng, L., Ehrhardt, M. and Gunther, M. (2014). The dynamic Correlation model and its application to the Heston model, Preprint 14/09, University of Wuppertal.

\bibitem{vanHaastrecht2011}
van Haastrecht, A., and Pelsser, A. (2011). Generic pricing of FX, inflation and stock options under stochastic interest rates and stochastic volatility. Quantitative Finance, 11(5), 665–691. doi: 10.1080/14697688.2010.504734

\bibitem{Van}
Van, H.,  Pelsser, A. (2011). Generic pricing of FX, inflation and stock options under stochastic interest rates and stochastic volatility. Quantitative Finance, 11(5):665-691.

\bibitem{Xi}Xi, F., Zhu C. (2019). Jump type stochastic differential equations with non-Lipschitz coefficients: non-confluence, Feller and strong Feller properties, and exponential ergodicity. Journal of Differential Equations 266.8: 4668-4711.

\bibitem{Wu}
We, L.  Kwok, Y. K. (1997).  A front-fixing finite difference method for the valuation of {A}merican options. J. Financial Engineering.
%


\end{thebibliography}
\end{document}